\documentclass{amsart}
\usepackage{hyperref}
\usepackage[all]{xy}
\usepackage{amscd}
\usepackage{amssymb}
\usepackage{graphicx}

\numberwithin{equation}{section}

\newtheorem{thm}[equation]{Theorem}
\newtheorem{lem}[equation]{Lemma}
\newtheorem{cor}[equation]{Corollary}
\newtheorem{prop}[equation]{Proposition}

\newtheorem*{mlem}{{\sc Lemma 2.1.4}}

\theoremstyle{definition}
\newtheorem{rem}[equation]{Remark}
\newtheorem{defn}[equation]{Definition}

\renewcommand{\labelenumi}{\theenumi}
\renewcommand{\theenumi}{(\roman{enumi})}

\newfont{\cyrr}{wncyr10}

\def\Z{\mathbf{Z}}
\def\Q{\mathbf{Q}}
\def\F{\mathbf{F}}
\def\R{\mathbf{R}}
\def\C{\mathbf{C}}

\def\Zp{\Z_p}
\def\Qp{\Q_p}
\def\Fp{\F_p}

\def\d{\mathfrak{d}}

\def\O{\mathcal{O}}

\def\FF{\mathcal{F}}

\def\HH{\mathcal{S}}

\def\DD{\mathcal{D}}

\def\II{\mathcal{I}}
\def\PP{\mathcal{P}}
\def\LL{\mathcal{L}}

\def\NN{\mathcal{N}}
\def\cS{\mathcal{S}}

\def\X{\mathcal{X}}

\def\cH{\mathcal{H}}

\def\a{\mathfrak{a}}
\def\b{\mathfrak{b}}
\def\c{\mathfrak{c}}
\def\m{\mathbf{m}}
\def\n{\mathfrak{n}}

\def\l{\mathfrak{q}}
\def\el{\mathfrak{l}}
\def\p{\mathfrak{p}}

\def\mm{\mathfrak{m}}

\def\k{\Bbbk}

\def\Hom{\mathrm{Hom}}
\def\Gal{\mathrm{Gal}}

\def\Aut{\mathrm{Aut}}

\def\Cor{\mathrm{Cor}}

\def\unr{\mathrm{ur}}
\def\ord{\mathrm{ord}}
\def\loc{\mathrm{loc}}

\def\dv{\mathrm{div}}

\def\image{\mathrm{image}}
\def\rank{\mathrm{rank}}
\def\corank{\mathrm{corank}}
\def\length{\mathrm{length}}

\def\Fr{\mathrm{Fr}}
\def\Frl{\Fr_{\l}}

\def\tr{{\mathrm{tr}}}

\def\GL{{\mathrm{GL}}}

\def\GSp{{\mathrm{GSp}}}

\def\Sel{{\mathrm{Sel}}}

\def\Maps{{\mathrm{Maps}}}

\def\Cl{\mathrm{Cl}}

\def\sat{\mathrm{sat}}
\def\ru{\mathrm{u}}

\def\f{\mathrm{f}}

\def\HS#1{H^1_{#1}}
\def\Hf{\HS{\f}}

\def\Ht{\HS{\tr}}

\def\HF{\HS{\FF}}

\def\HFs{\HS{\FF^*}}

\def\HFfn{\HS{\FFf{\n}}}
\def\HFfnl{\HS{\FFf{\n\l}}}
\def\HFfns{\HS{\FFf{\n}^*}}
\def\HFfnls{\HS{\FFf{\n\l}^*}}

\def\SSsrf#1#2#3#4{#1_{#2}^{#3}(#4)}

\def\Fsrf#1#2#3{\SSsrf{\FF}{#1}{#2}{#3}}

\def\FFf#1{\Fsrf{}{}{#1}}

\def\shmap#1#2{\psi_{#1}^{#2}}
\def\pve{\shmap{v}{e}}
\def\pvpe{\shmap{v'}{e}}
\def\pne{\shmap{\n}{e}}
\def\pnle{\shmap{\n\l}{e}}

\def\map#1{\;\xrightarrow{#1}\;}

\def\isom{\map{\sim}}
\def\isomm{\map{\;\sim\;}}

\def\bmu{\boldsymbol\mu}
\def\ld{{\mathcal h}}
\def\rd{{\mathcal i}}
\def\too{\longrightarrow}
\def\hookto{\hookrightarrow}
\def\onto{\twoheadrightarrow}

\def\Il{I_\l}
\def\In{I_\n}
\def\Inl{I_{\n\l}}
\def\fs{\phi^{\mathrm{fs}}}
\def\fsl{\fs_{\l}}

\def\ssys{\boldsymbol{\epsilon}}
\def\ksys{\boldsymbol{\kappa}}

\def\KS{\mathbf{KS}}

\def\SS{\mathbf{SS}}
\def\PPk{\PP_k}
\def\NNk{\NN_k}

\def\pair#1#2{\ld #1, #2 \rd}
\def\pairl#1#2{{\pair{#1}{#2}}_\l}

\def\mesh{\lambda}
\def\smesh{\mu}

\def\dirsum#1{\underset{#1}{\textstyle\bigoplus}}
\def\tensor#1{\underset{#1}{\textstyle\bigotimes}}
\def\Tb{{\bar{T}}}

\def\Gl{G_\l}
\def\Gn{G_\n}
\def\Gnl{G_{\n\l}}

\def\cm{\chi}

\def\der#1{\partial^{(#1)}}

\def\FFc{\FF_{\mathrm{ur}}}

\def\teich{\omega}

\def\psiK1{\psi_1}

\def\ssyi#1{\ssys^{(#1)}}
\def\kapi#1{\epsilon^{(#1)}}

\def\ordv#1{\ord(#1)}

\def\Kb{{\bar{K}}}
\def\Kl{K_\l}
\def\Klb{\shortoverline{\Kl}}

\def\W#1{W_{#1}}

\def\Y#1{Y_{#1}}
\def\divity{\varphi}

\def\invlim#1{{\underset{{\raise .5ex\hbox{$\scriptscriptstyle #1$}}}{\varprojlim}}\,}
\def\dirlim#1{{\underset{{\raise .5ex\hbox{$\scriptscriptstyle #1$}}}{\varinjlim}}\,}
\def\shortoverline#1{\hskip 2pt \overline{\hskip -2pt #1 \hskip -2pt} \hskip 2pt}

\def\dual#1{#1^{\bullet}}

\begin{document}

\title[Controlling Selmer groups in the higher core rank case]{Controlling 
   Selmer groups in the \\ higher core rank case}

\author{Barry Mazur}
\address{Department of Mathematics,
Harvard University,
Cambridge, MA 02138 USA}
\email{mazur@math.harvard.edu}

\author{Karl Rubin}
\address{Department of Mathematics,
UC Irvine,
Irvine, CA 92697 USA}
\email{krubin@math.uci.edu}

\date{\today}

\subjclass[2010]{Primary 11G40, 11F80; Secondary 11R23, 11R34}

\thanks{This material is based upon work supported by the 
National Science Foundation under grants DMS-1302409 and DMS-1065904.}

\begin{abstract}
We define Kolyvagin systems and Stark systems attached to $p$-adic representations 
in the case of arbitrary ``core rank'' (the core rank is a measure of 
the generic Selmer rank in a family of Selmer groups).  Previous work dealt only with the 
case of core rank one, where the Kolyvagin and Stark systems are collections of cohomology 
classes.  For general core rank, they are collections of elements of exterior powers 
of cohomology groups.  We show under mild hypotheses that for general core rank these 
systems still control the size and structure of Selmer groups, and that the 
module of all Kolyvagin (or Stark) systems is free of rank one.
\end{abstract}

\maketitle

\tableofcontents

\section*{Introduction}

Let $K$ be a number field and $G_K:= {\rm Gal}({\bar K}/K)$ its Galois group. 
Let $R$ be either a principal artinian local ring, or a discrete valuation ring, 
and $T$ an $R[G_K]$-module that 
is free over $R$ of finite rank. Let  $T^*:= \Hom(T,\bmu_\infty)$ be its 
Cartier dual. 

A cohomology class  $c$  in $H^1(G_K,T)$ 
provides (after localization and cup-product)  a linear functional $\LL_{c,v}$ 
on $H^1(G_{K_v},T^*)$ 
for any place $v$ of $K$. Thanks to the duality theorems of class field theory, these 
$\LL_{c,v}$, when summed over all places $v$ of $K$, give a linear functional 
$\LL_{c}$ that annihilates the adelic image of $H^1(G_{K},T^*)$.  
By imposing local conditions on the class $c$, we get a linear functional 
that annihilates a Selmer group in $H^1(G_{K},T^*)$.  
Following this thread, a systematic construction of classes $c$ can 
be of used to control the size of Selmer groups. Even better, a sufficiently 
full collection (a {\em system})  of classes $c$ can sometimes be used to completely 
determine the structure of the relevant Selmer groups. 

We have just described  a very vague outline of the strategy of controlling 
Selmer groups of Galois representations $T^*$, by systems of cohomology classes 
for $T$. 
In practice there are variants of this strategy.  First, we will control the local conditions 
that we impose on our cohomology classes.  That is, we will require our classes 
to lie in certain Selmer groups for $T$.
But more importantly, in general one encounters situations where sufficiently many 
of the relevant Selmer groups for $T$ are free over $R$ of some (fixed) rank $r \ge 1$.  
We call $r$ the {\em core rank} of $T$; see Definition \ref{crd} below.  
In the natural cases that we consider, all relevant Selmer groups contain a 
free module of rank equal to the core rank $r$, and $r$ is maximal with respect 
to this property.

If $R$ is a discrete valuation ring and 
our initial local conditions are what we call {\em unramified} 
(see Definition \ref{css} and Theorem \ref{dvrdm}), 
then under mild hypotheses the core rank $r$ of $T$ is given by the simple formula 
$$
r = \sum_{v\mid\infty} \corank\, H^0(G_{K_v}, T^*).
$$  
So, for example, if $T$ is the $p$-adic Tate module of 
an abelian variety of dimension $d$ over $K$, then the  
core rank is $d\,[K:{\Q}]$.

To deal with the case where $r$ is greater than $1$ we will ask for elements in the 
$r$-th exterior powers (over $R$) of those Selmer groups, so that for every $r$
we will be seeking systems of classes in $R$-modules that are often free of rank 
one over $R$.

One of the main aims of this article is to extend the more established 
theory of core rank $r=1$ (see for example \cite{kolysys}) to the case of 
higher core rank. We deal with two types of systems of cohomology classes: 
{\em Stark systems} (collections of classes generalizing the 
units predicted by Stark-type conjectures) 
and {\em Kolyvagin systems} (generalizing Kolyvagin's original formulation). 
Our Stark systems are similar to the ``unit systems''
that occur in the recent work of Sano \cite{sano}.
There is a third type, {\em Euler systems} (see for example \cite{pr.es} or \cite{EulerSystems}), 
which we do not deal with in this paper.
When $r=1$, Euler systems provide the crucial link (\cite[Theorem 3.2.4]{kolysys}) 
between Kolyvagin or Stark systems and $L$-values.  We expect that when $r>1$ 
there is still a connection between Euler systems on the one hand, and 
Stark and Kolyvagin systems on the other, but this connection is still mysterious.
For an example of the sort of connection that we expect, see the forthcoming paper 
\cite{gen.darmon}.


The Euler systems that have been already constructed in the literature, 
or that are conjectured to exist, are motivic: they come from arithmetic objects 
such as  circular units or more generally the conjectural Stark units; 
or---in another context---Heegner points; or elements of $K$-theory.  
Euler systems  are `vertically configured' in the sense that they provide 
classes in many abelian extensions of the base number field, and the classes cohere 
via norm projection from one abelian extension to a smaller one when modified 
by the multiplication of appropriate `Euler factors' (hence the terminology `Euler system'). 

On the other hand, the Stark and Kolyvagin systems 
are `horizontally configured'  in the sense that they consist only 
of cohomology classes over the base number field, but conform to a 
range of local conditions. 
The local conditions for Stark systems are more elementary 
and---correspondingly---the Stark systems are somewhat easier to handle 
than Kolyvagin systems. In contrast, the local conditions for Kolyvagin 
systems connect more directly with the changes of local conditions that 
arise from twisting the Galois representation $T$ by characters. 
 
One of the main results of this paper (Theorem \ref{mthm3'}) is that---under suitable 
hypotheses, but for general core rank---there is an equivalence 
between Stark systems and special Kolyvagin systems that we call 
{\em stub Kolyvagin systems}, and, up to a scalar unit, there is a 
unique `best' Stark (equivalently: stub Kolyvagin) system 
(Theorems \ref{magic} and \ref{rankdvr}). We show, 
as mentioned in the title of this article, that the corresponding Selmer 
modules are controlled by (either of) these systems 
(Theorems \ref{korank} and \ref{korankk}), 
in the sense that there is a relatively simple 
description of the elementary divisors (and hence the isomorphy type) 
of the Selmer group of $T$ starting with any Stark or stub Kolyvagin system.  
When the core rank is one, every Kolyvagin system is a stub Kolyvagin 
system \cite[Theorem 4.4.1]{kolysys}.

 
  
Although we have restricted our scalar rings $R$ to be either principal 
artinian local rings or complete discrete valuation rings 
with finite residue field, it is natural to wish to extend the format of 
our systems of cohomology classes to encompass Galois representations $T$ 
that are free of finite rank over more general complete local rings, so as 
to be able to deal effectively with deformational questions. 

\subsection*{Layout of the paper}
In Part \ref{part1} (sections \ref{lcg}--\ref{esexs}) we recall basic facts that we will 
need about local and global cohomology groups, and define our abstract 
Selmer groups and the core rank.
In Part \ref{part2} (sections \ref{Ssec}--\ref{extra}) 
we define Stark systems and investigate the relations 
between Stark systems 
and the structure of Selmer groups.  
Part \ref{part3} 
(sections \ref{sog}--\ref{m0pf}) deals with Kolyvagin systems, and the relation between 
Kolyvagin systems and Stark systems.

The results of \cite{kolysys} were restricted to the case where the base field $K$ is $\Q$.  
In many cases the proofs for general $K$ are the same, and in those cases we will 
feel free to use results from \cite{kolysys} without further comment.

\subsection*{Notation}

Fix a rational prime $p$.  Throughout this paper, $R$ will denote a complete,
noetherian, local principal ideal domain with finite residue field of characteristic $p$.  
Let $\m$ denote the maximal ideal of $R$. 
The basic cases to keep in mind are $R = \Z/p^n\Z$ or $R = \Zp$.

If $K$ is a field, $\Kb$ will denote a fixed separable closure of $K$ and 
$G_K := \Gal(\Kb/K)$.
If $A$ is an $R$-module and $I$ is an ideal of $R$, we will write 
$A[I]$ for the submodule of $A$ killed by $I$.  If $A$ is a $G_K$-module, 
we write $K(A)$ for the fixed field in $\Kb$ of the kernel of the map
$G_K \to \Aut(A)$. 

If a group $H$ acts on a set $X$, then the subset of
elements of $X$ fixed by $H$ is denoted $X^H$.

If $n$ is a positive integer, $\bmu_n$ will denote the group of 
$n$-th roots of unity in $\Kb$.

\part{Cohomology groups and Selmer structures}
\label{part1}

\section{Local cohomology groups}
\label{lcg}

For this section
$K$ will be a local field (archimedean or nonarchimedean).  
If $K$ is nonarchimedean let
$\O$ be the ring of integers in $K$, $\F$ its residue field,
$K^{\unr}\subset {\bar K}$ the maximal unramified subfield of ${\bar K}$, and 
$\II$ the inertia group $\Gal({\bar K}/K^{\unr})$, so 
$G_\F = G_K/\II = \Gal(K^{\unr}/K)$.  

Fix an $R$-module $T$ endowed with a continuous
$G_K$-action. By $H^*(K,T) := H^*(G_K,T)$ we mean cohomology computed 
with respect to continuous cochains.  

\begin{defn}
\label{locconddef}
A {\em local condition} on $T$ (over $K$) is a choice of an $R$-sub\-module 
of $H^1(K,T)$.  If we refer to the local condition by a symbol, say $\FF$, 
we will denote the corresponding $R$-submodule
$
\HF(K,T) \subset H^1(K,T).
$ 

If $I$ is an ideal of $R$, then a local condition on 
$T$ induces local conditions on $T/IT$ and $T[I]$
by taking $\HF(K,T/IT)$ and $\HF(K,T[I])$ to be the 
image and inverse image, respectively, of $\HF(K,T)$ 
under the maps induced by
$$
T \onto T/IT, \qquad T[I] \hookto T.
$$
One can similarly propagate the local condition $\FF$ canonically to arbitrary 
subquotients of $T$, and 
if $R \to R'$ is a homomorphism of complete noetherian local 
PID's, then $\FF$ induces a local condition on the $R'$-module 
$T \otimes_R R'$.
\end{defn}

\begin{defn}
\label{conds}
Suppose $K$ is nonarchimedean and $T$ is unramified (i.e., $\II$ acts trivially on $T$).
Define the {\em finite} (or {\em unramified}) local condition by
$$
\Hf(K,T) := \ker\bigl[H^1(K,T) \to H^1(K^{\unr},T)\bigr] = H^1(K^\unr/K,T). 
$$

More generally, if $L$ is a Galois extension of $K$ we define the {\em $L$-transverse} 
local condition by
$$
\HS{\text{$L$-$\tr$}}(K,T) := \ker\bigl[H^1(K,T) \to H^1(L,T)\bigr] = H^1(L/K,T^{G_L}).
$$
\end{defn}

Suppose for the rest of this section that the local field $K$ is nonarchimedean, 
the $R$-module $T$ is of finite type, and the action of 
$G_K$ on $T$ is unramified.

Fix a totally tamely ramified cyclic extension $L$ of $K$ such that 
$[L:K]$ annihilates $T$.  We will write simply $\Ht(K,T)$ for 
$\HS{\text{$L$-$\tr$}}(K,T) \subset H^1(K,T)$.

\begin{lem}
\label{isofunct}
\begin{enumerate}
\item
The composition 
$$
\Ht(K,T) \hookto H^1(K,T) \onto H^1(K,T)/\Hf(K,T)
$$ 
is an isomorphism, so there is a canonical splitting
$$
H^1(K,T) = \Hf(K,T) \oplus \Ht(K,T).
$$
\end{enumerate}
There are canonical functorial isomorphisms
\begin{enumerate}
\addtocounter{enumi}{1}
\item
$\Hf(K,T) \cong T/(\Fr-1)T$,
\item
$\Ht(K,T) \cong \Hom(\II,T^{\Fr=1})$, \quad
$\Ht(K,T) \otimes \Gal(L/K) \cong T^{\Fr=1}$.
\end{enumerate}
\end{lem}

\begin{proof}
Assertion (i) is \cite[Lemma 1.2.4]{kolysys}.
The rest is well known; see for example \cite[Lemma 1.2.1]{kolysys}.
\end{proof}

\begin{defn}
\label{fsmap}
Suppose that $T$ is free of finite rank as an $R$-module, and that 
$\det(1-\Fr \;|\; T) = 0$.
Define $P(x) \in R[x]$ by 
$$
P(x) := \det(1-\Fr\;x \;|\; T).
$$ 
Since $P(1) = 0$, there is a unique polynomial $Q(x) \in R[x]$ such that 
$$
(x-1)Q(x) = P(x) \quad \text{in $R[x]$}.
$$
By the Cayley-Hamilton theorem, $P(\Fr^{-1})$ annihilates 
$T$, so $Q(\Fr^{-1}) T \subset T^{\Fr=1}$.
We define the {\em finite-singular comparison map} $\fs$ on $T$ to be 
the composition, using the isomorphisms of Lemma \ref{isofunct}(ii,iii),
$$
\Hf(K,T) \isom T/(\Fr-1)T \map{Q(\Fr^{-1})}
    T^{\Fr=1} \isom \Ht(K,T) \otimes \Gal(L/K).
$$
\end{defn}

\begin{lem}
\label{applemma}
Suppose that $T$ is free of finite rank over $R$, and 
that $T/(\Fr-1)T$ is a free $R$-module of rank one.
Then $\det(1-\Fr \;|\; T) = 0$ and the map
$$
\fs : \Hf(K,T) \too \Ht(K,T) \otimes \Gal(L/K)
$$
of Definition \ref{fsmap} is an isomorphism.  In particular both
$\Hf(K,T)$ and $\Ht(K,T)$ are free of rank one over $R$.
\end{lem}

\begin{proof}
This is \cite[Lemma 1.2.3]{kolysys}.
\end{proof}

\begin{defn}
\label{locdualdef}
Define the {\em dual} of $T$ to be the $R[[G_K]]$-module
$$
T^* := \Hom(T,\bmu_{p^\infty}).
$$ 
We have the (perfect) local Tate cup product pairing
$$
\pair{~}{~} : H^1(K,T) \times H^1(K,T^*) 
    \too H^2(K,\bmu_{p^\infty}) \isomm \Qp/\Zp.
$$

A local condition $\FF$ for $T$ determines a local condition $\FF^*$
for $T^*$, by taking $\HFs(K,T^*)$ to be the orthogonal
complement of $\HF(K,T)$ under the Tate pairing $\pair{~}{~}$.
\end{defn}

\begin{prop}
\label{h1fdual}
With notation as above, we have:
\begin{enumerate}
\item
$\Hf(K,T)$ and $\Hf(K,T^*)$ are orthogonal complements under 
$\pair{~}{~}$.
\item
$\Ht(K,T)$ and $\Ht(K,T^*)$ are orthogonal 
complements under $\pair{~}{~}$.
\end{enumerate}
\end{prop}

\begin{proof}
The first assertion is (for example) Theorem I.2.6 of \cite{milne}.
Both assertions are \cite[Lemma 1.3.2]{kolysys}.
\end{proof}

\section{Global cohomology groups and Selmer structures}
\label{gcg}

For the rest of this paper, $K$ will be a number field and 
$T$ will be a finitely generated free $R$-module
with a continuous action of $G_K$, that is unramified outside a finite 
set of primes.  

{\em Global notation.}
Let $\Kb \subset \C$ be the  algebraic closure of $K$ in $\C$, 
and for each prime $\l$ of $K$ fix an algebraic closure  
$\Klb$ of $\Kl$ containing $\Kb$.  This determines a choice of extension 
of $\l$ to $\Kb$.
Let $\DD_\l := \Gal(\Klb/\Kl)$, which we identify with a closed subgroup of
$G_K := \Gal(\Kb/K)$.  In other words $\DD_\l$ is a particular 
decomposition group at $\l$ in $G_K$, and $H^1(\DD_\l,T) = H^1(\Kl,T)$. 
Let $\II_\l \subset \DD_\l$ be the inertia group, and
$\Frl \in \DD_\l/\II_\l$  the Frobenius element.  If $T$
is unramified at $\l$, then
$\DD_\l/\II_\l$ acts on $T$, and hence so does $\Frl$.  
If we choose a different decomposition group at
$\l$, then the action of $\Frl$ changes by conjugation in $G_K$.
We will write $\loc_\l$ for the localization map $H^1(K,T) \to H^1(\Kl,T)$.  

If $\l$ is a prime of $K$, let $K(\l)$ denote the $p$-part of the ray class field of $K$ 
modulo $\l$ (i.e., the maximal $p$-power extension of $K$ in the ray class field), 
and $K(\l)_\l$ the completion of $K(\l)$ at the chosen prime above $\l$.  
If $\l$ is principal then $K(\l)_\l/K_\l$ is cyclic and totally tamely ramified.

If $\l$ is principal, $T$ is unramified at $\l$, and $[K(\l)_\l:K_\l]T = 0$, 
the {\em transverse sub\-module} of $H^1(\Kl,T)$ is the submodule 
$$
\Ht(\Kl,T) := \HS{\text{$K(\l)_\l$-$\tr$}}(\Kl,T) 
   = \ker\bigl[H^1(\Kl,T) \to H^1(K(\l)_\l,T)\bigr]
$$ 
of Definition \ref{conds}.

\begin{defn}
\label{selmerdef}
A {\em Selmer structure} $\FF$ on $T$ is a collection of the following data:
\begin{itemize}
\item
a finite set $\Sigma(\FF)$ of places of $K$, including all infinite places, 
all primes above $p$, and all primes where $T$ is ramified,
\item
for every $\l \in \Sigma(\FF)$ (including archimedean places), a local condition 
(in the sense of Definition \ref{locconddef}) on $T$ over $K_\l$, i.e., a choice of $R$-submodule
$
\HF(\Kl,T) \subset H^1(\Kl,T).
$
\end{itemize}

If $\FF$ is a Selmer structure, we define the 
{\em Selmer module} $\HF(K,T) \subset H^1(K,T)$ to be the kernel
of the sum of restriction maps
$$
H^1(K_{\Sigma(\FF)}/K,T) \longrightarrow
    \dirsum{\l\in\Sigma(\FF)}\left(H^1(\Kl,T)/\HF(\Kl,T)\right)
$$
where $K_{\Sigma(\FF)}$ denotes the maximal extension of $K$ that is
unramified outside $\Sigma(\FF)$.  In other words, $\HF(K,T)$ consists
of all classes which are unramified (or equivalently, finite) outside of 
$\Sigma(\FF)$ and which locally at $\l$ belong to $\HF(\Kl,T)$ for 
every $\l \in \Sigma(\FF)$.
\end{defn}

For examples of Selmer structures see \cite{kolysys}.  Note that 
if $\FF$ is a Selmer structure on $T$ and $I$ is an ideal of $R$, 
then $\FF$ induces canonically (see Definition \ref {locconddef}) 
Selmer structures on the $R/I$-modules $T/IT$ and $T[I]$, 
that we will also denote by $\FF$.

\begin{defn}
\label{pndef}
Suppose now that $T$ is free over $R$, 
$\l \nmid p\infty$ is prime, and $T$ is unramified at $\l$.  
If $\l$ is not principal, let $\Il := R$.
If $\l$ is principal, let $\Il \subset R$ be the largest power of $\m$ 
(i.e., $\m^k$ with $k \ge 0$ maximal) 
such that $[K(\l)_\l:\Kl]R \subset \Il$ and $T/((\Frl-1)T+\Il T)$ is free of rank one 
over $R/\Il$.

Let $\PP$ denote a set of prime ideals of $K$, disjoint from $\Sigma(\FF)$.
Typically $\PP$ will be a set of positive density.
Define a filtration $\PP \supset \PP_1 \supset \PP_2 \supset \cdots$ by
$$
\PPk = \{\l \in \PP : I_\l \subset \m^k\}
$$
for $k \ge 1$.  
Let $\NN := \NN(\PP)$ denote the set of squarefree products
of primes in $\PP$ (with the convention that the trivial ideal $1 \in \NN$).  
Let $I_1 := 0$ and if $\n \in \NN$, $\n \ne 1$, define
$$
\In := \sum_{\l \mid \n} \Il \subset R.
$$
\end{defn}

\begin{defn}
\label{FFabc}
Suppose $\FF$ is a Selmer structure, and $\a, \b, \n$ are pairwise relatively prime 
ideals of $K$ with $\n \in\NN$ and $\In T = 0$.
Define a new Selmer structure $\Fsrf{\a}{\b}{\n}$ by
\begin{itemize}
\item
$\Sigma(\Fsrf{\a}{\b}{\n}) := \Sigma(\FF) \cup \{\l : \l \mid \a\b\n\}$,
\item
$
\HS{\Fsrf{\a}{\b}{\n}}(\Kl,T) := \begin{cases}
\HF(\Kl,T) & \text{if $\l \in \Sigma(\FF)$},\\
0 & \text{if $\l \mid \a$}, \\
H^1(\Kl,T) & \text{if $\l \mid \b$}, \\
\Ht(\Kl,T) & \text{if $\l \mid \n$}.
\end{cases}
$
\end{itemize}
In other words, $\Fsrf{\a}{\b}{\c}$ consists of $\FF$ together with 
the strict condition at 
primes dividing $\a$, the unrestricted condition at primes dividing $\b$, 
and the transverse condition at primes dividing $\n$.
\end{defn}

If any of $\a, \b, \n$ are the trivial ideal, we may suppress them 
from the notation.  For example, we will be especially interested in 
Selmer groups of the form
\begin{align*}
\HS{\FF^\n}(K,T) &: \text{no restriction at $\l$ dividing $\n$, same as $\FF$ elsewhere,} \\
\HFfn(K,T/\In T) &: \text{transverse condition at $\l$ dividing $\n$, same as $\FF$ elsewhere.}
\end{align*}
If $\mm\mid\n \in \NN$, the definition leads to an exact sequence
\begin{equation}
\label{b}
0 \too \HS{\FF^\mm}(K,T) \too \HS{\FF^\n}(K,T) \too 
   \dirsum{\l\mid(\n/\mm)} H^1(K_\l,T)/\Hf(K_\l,T).
\end{equation}

\begin{defn}
\label{globdualdef}
The {\em dual} of $T$ is the $R[[G_K]]$-module
$
T^* := \Hom(T,\bmu_{p^\infty}).
$
For every $\l$ we have the local Tate pairing
$$
\pairl{~}{~} : H^1(\Kl,T) \times H^1(\Kl,T^*) \too \Qp/\Zp
$$
as in \S\ref{lcg}.

Just as every local condition on $T$ determines a local condition 
on $T^*$ (Definition \ref{locdualdef}), 
a Selmer structure $\FF$ for $T$ determines a Selmer structure $\FF^*$
for $T^*$.  Namely, take $\Sigma(\FF^*) := \Sigma(\FF)$, and for
$\l \in \Sigma(\FF)$ take $\HFs(\Kl,T^*)$ to be the local condition 
induced by $\FF$, i.e., the orthogonal complement of $\HF(\Kl,T)$ under $\pairl{~}{~}$.
\end{defn}

\section{Selmer structures and the core rank}
\label{notsscr}

Suppose for this section that the $R$ is a principal local ring.
We continue to assume for the rest of this paper that $T$ is free of finite rank over $R$, 
in addition to being a $G_K$-module. 

\begin{defn}
\label{cardef}
A Selmer structure $\FF$ on $T$ is is called {\em cartesian} if 
for every $\l \in \Sigma(\FF)$, the local condition $\FF$ at $\l$ 
is ``cartesian on the category of quotients of $T$'' as defined in 
\cite[Definition 1.1.4]{kolysys}.
\end{defn}

\begin{rem}
\label{5.2}
If $\FF$ is cartesian then 
for every $k$ the induced Selmer structure on the 
$R/\m^k$-module $T/\m^kT$ is cartesian.  
If $R$ is a field (i.e., $\m = 0$) then every Selmer structure on $T$ is cartesian.  
If $R$ is a discrete valuation ring and $H^1(\Kl,T)/\HF(\Kl,T)$ is torsion-free
for every $\l \in \Sigma(\FF)$, then $\FF$ is cartesian (see \cite[Lemma 3.7.1(i)]{kolysys}).
\end{rem}

\begin{prop}
\label{stprops}
Suppose $R$ is a principal artinian local ring of length $k$ (i.e., 
$\m^k = 0$ and $\m^{k-1} \ne 0$), $\FF$ is a cartesian Selmer structure on $T$, 
and $T^{G_K} = (T^*)^{G_K} = 0$.

If $\n\in\NN$ and $\In = 0$ then:
\begin{enumerate}
\item
the exact sequence 
$$
0 \too T/\m^i T \too T \too T/\m^{k-i}T \to 0
$$ 
induces an isomorphism $\HFfn(K,T/\m^i T) \isom \HFfn(K,T)[\m^i]$ 
and an exact sequence
$$
0 \too \HFfn(K,T)[\m^i] \too \HFfn(K,T) \too \HFfn(K,T/\m^{k-i}T).
$$
\item
the inclusion $T^*[\m^i] \hookto T^*$ induces an isomorphism 
$$
\HFfns(K,T^*[\m^i]) \isomm \HFfns(K,T^*)[\m^i].
$$
\item
there is a unique integer $r$, independent of $\n$, such that there is a noncanonical isomorphism
$$
\renewcommand{\arraystretch}{1.5}
\arraycolsep=2pt
\begin{array}{rcll}
\HFfn(K,T) &\cong& \HFfns(K,T^*) \oplus R^r & \text{~if $r \ge 0$}, \\
\HFfn(K,T) \oplus R^{-r} &\cong& \HFfns(K,T^*) & \text{~if $r \le 0$}.
\end{array}
$$
\end{enumerate}
\end{prop}

\begin{proof}
These assertions are \cite[Lemma 3.5.4]{kolysys},  \cite[Lemma 3.5.3]{kolysys}, and 
\cite[Theorem 4.1.5]{kolysys}, respectively.
\end{proof}

\begin{defn}
\label{crd}
Suppose $\FF$ is a cartesian Selmer structure on $T$.  If $R$ is artinian, then 
the {\em core rank} of $(T,\FF)$ is the integer $r$ of Proposition \ref{stprops}(iii).  
If $R$ is a discrete valuation ring, then the core rank of $(T,\FF)$ is 
the core rank of $(T/\m^k T,\FF)$ for 
every $k > 0$, which by Proposition \ref{stprops} is independent of $k$.

We will denote the core rank by $\cm(T,\FF)$, or simply $\cm(T)$ when $\FF$ is understood.
\end{defn}

For $\n\in\NN$, let $\nu(\n)$ denote the number of primes dividing $\n$.

\begin{cor}
\label{threefive} 
Suppose $R$ is artinian, $\cm(T) \ge 0$, $\n\in\NN$, and $\In = 0$.  
Let $\lambda(\n) := \length(\HS{(\FF(\n)^*}(K,T^*))$ and 
$\mu(\n) := \length(\HS{(\FF^*)_{\n}}(K,T^*))$.  
There are noncanonical isomorphisms
\begin{enumerate}
\item
$\HFfn(K,T) \cong \HFfns(K,T^*) \oplus R^{\cm(T)}$,
\item
$\HS{\FF^\n}(K,T) \cong \HS{(\FF^*)_{\n}}(K,T^*) \oplus R^{\cm(T)+\nu(\n)}$, 
\item
$\m^{\lambda(\n)}\wedge^{\cm(T)}\HFfn(K,T) \cong \m^{\lambda(\n)}$, 
\item
$\m^{\mu(\n)}\wedge^{\cm(T)+\nu(\n)}\HS{\FF^\n}(K,T) \cong \m^{\mu(\n)}$.
\end{enumerate}
\end{cor}

\begin{proof}
The first isomorphism is just Proposition \ref{stprops}(iii).  For (ii), 
observe that the Selmer structure $\FF^\n$ is 
cartesian by \cite[Lemma 3.7.1(i)]{kolysys}, so 
applying Proposition \ref{stprops}(iii) to $(T,\FF^\n)$ we have 
$
\HS{\FF^\n}(K,T) \cong \HS{(\FF^*)_{\n}}(K,T^*) \oplus R^{\cm(T,\FF^\n)}.
$
To complete the proof of (ii) 
we need only show that $\cm(T,\FF^\n) = \cm(T)+\nu(\n)$, 
and this follows without difficulty from Poitou-Tate global duality 
(see for example \cite[Theorem 2.3.4]{kolysys}).

Assertions (iii) and (iv) follow directly from (i) and (ii), respectively.
\end{proof}

\section{Running hypotheses}
\label{sscr}

\begin{defn}
\label{sddef}
By {\em Selmer data} we mean a tuple $(T,\FF,\PP,r)$ where 
\begin{itemize}
\item
$T$ is a $G_K$-module, free of finite rank over $R$, 
unramified outside finitely many primes, 
\item
$\FF$ is a Selmer structure on $T$,
\item
$\PP$ is a set of primes of $K$ disjoint from $\Sigma(\FF)$,
\item
$r \ge 1$.
\end{itemize}
\end{defn}

\begin{defn}
\label{mide}
If $L$ is a finite Galois extension of $K$ and $\tau \in G_K$, define 
\begin{multline*}
\PP(L,\tau) := \{\text{primes $\l \notin \Sigma(\FF)$ : $\l$ is unramified in $L/K$} \\ 
    \text{and $\Frl$ is conjugate to $\tau$ in $\Gal(L/K)$}\}.
\end{multline*}
\end{defn}

Fix Selmer data $(T,\FF,\PP,r)$ as in Definition \ref{sddef}.  Let $\Tb = T/\m T$, so $\Tb^* = T^*[\m]$.
If $R$ is artinian, let $M$ denote the smallest power of $p$ such that $MR = 0$.  If $R$ is a discrete 
valuation ring, let $M:= p^\infty$.
Let $\cH$ denote the Hilbert class field of $K$, and $\cH_M := \cH(\bmu_M,(\O_K^\times)^{1/M})$.
Let $\k$ denote the residue field $R/\m$.
In order to obtain the strongest results, we will usually make the following additional assumptions.

{\renewcommand{\labelenumi}{\theenumi}
\renewcommand{\theenumi}{(H.\arabic{enumi})}
\begin{enumerate}
\item
\label{h.1}
$\Tb^{G_K} = (\Tb^*)^{G_K} = 0$ and 
$\Tb$ is an absolutely irreducible $\k[[G_K]]$-module, 
\item
\label{h.2} 
there is a $\tau \in \Gal(\Kb/\cH_M)$ and a finite Galois extension $L$ of $K$ in $\cH_M$
such that $T/(\tau-1)T$ is free of rank one over $R$ and $\PP(L,\tau) \subset \PP$,
\item
\label{h.3}
$H^1(\cH_M(T)/K,T/\m T) = H^1(\cH_M(T)/K,T^*[\m]) = 0$,
\item
\label{h.4}
either $\Tb \not\cong \Tb^*$ as $\k[[G_K]]$-modules, or $p > 3$,
\item
\label{h.5}
the Selmer structure $\FF$ is cartesian (Definition \ref{cardef}),
\item
\label{h.6}
$r = \cm(T) > 0$, where $\cm(T)$ is the core rank of $T$.
\end{enumerate}
(Only) when $R$ is artinian, we will also sometimes assume 
\begin{enumerate}
\addtocounter{enumi}{6}
\item
\label{h.7}
$\Il = 0$ for every $\l \in \PP$.
\end{enumerate}
}

\begin{rem}
\label{5.7}
Note that if the above properties hold for $(T,\FF,\PP,r)$, then they also 
hold if $R$ is replaced by $R/\m^k$ and $T$ by $T/\m^k$, for $k \ge 0$.  
If $R$ is artinian and \ref{h.1} through \ref{h.6} hold, then Lemma \ref{okl} below shows that 
\ref{h.1} through \ref{h.7} hold if we replace $L$ by $\cH_M$ and  $\PP$ by $\PP(\cH_M,\tau)$.
\end{rem}

\begin{rem}
Assumption \ref{h.5} is needed to have a well-defined notion of core rank.  
Assumption \ref{h.2} is needed to provide is with a large selection of 
primes $\l$ such that $T/(\Frl-1,\m^k)$ is free of rank one, for large $k$.
 
We deduce from assumption \ref{h.3} that restriction 
from $K$ to $\cH_M(T)$ is injective on the Selmer group; this allows us to view Selmer classes 
in $\Hom(G_{\cH_M(T)},T)$.  Assumptions \ref{h.1} and \ref{h.4} then allow 
us to satisfy various Cebotarev conditions simultaneously.
\end{rem}

\begin{lem}
\label{okl}
Suppose $R$ is artinian and $\tau$ is as in {\rm \ref{h.2}}.
If $\l \in \PP(\cH_M,\tau)$, then $\Il = 0$. 
\end{lem}

\begin{proof}
Since $\Frl$ fixes $\cH$, $\l$ is principal.  By class field theory we have 
\begin{equation}
\label{2.4.1}
\Gal(K(\l)_\l/\Kl) \cong (\O_K/\l)^\times/\image(\O_K^\times).
\end{equation}
Since $\tau$ acts trivially on $\bmu_M$, so does $\Frl$, so $|(\O_K/\l)^\times|$ is cyclic 
of order divisible by $M$.
Since $\tau$ acts trivially on $(\O_K^\times)^{1/M}$, so does $\Frl$, so the 
reduction of $\O_K^\times$ is contained in $((\O_K/\l)^\times)^M$.
By \eqref{2.4.1} we conclude that $[K(\l)_\l:\Kl]$ is divisible by $M$, 
so $[K(\l)_\l:\Kl]R = 0$. 
We also have that $T/(\Frl-1)T \cong T/(\tau-1)T$ is free of rank one over $R$, 
so the lemma follows from the definition of $\Il$.
\end{proof}

\section{Examples}
\label{esexs}

\subsection{A canonical Selmer structure}

\begin{defn}
\label{css}
When $R$ is a discrete valuation ring, 
we define a canonical {\em unramified Selmer structure} $\FFc$ on $T$ by
\begin{itemize}
\item
$\Sigma(\FFc) := \{\l : \text{$T$ is ramified at 
$\l$}\} \cup \{\p : \p \mid p\} \cup \{v : v \mid \infty\}$,
\item
if $\l \in \Sigma(\FFc)$ and $\l \nmid p\infty$ then 
$$
\HS{\FFc}(K_\l,T) := \ker\bigl[H^1(K_\l,T) \to H^1(K_\l^{\unr},T \otimes \Qp)\bigr],
$$
\item
if $\p \mid p$ then define the universal norm subgroup 
$$
H^1(K_\p,T)^\ru := \cap_{K_\p \subset L \subset K_\p^\unr} \Cor_{L/K_\p}H^1(L,T),
$$ 
intersection over all finite unramified extensions $L$ of $K_\p$.  Define
$$
\HS{\FFc}(K_\p,T) := H^1(K_\p,T)^{\ru,\sat},
$$
the saturation of 
$H^1(K_\p,T)^\ru$ in $H^1(K_\p,T)$, i.e., 
$H^1(K_\p,T)/\HS{\FFc}(K_\p,T)$ is $R$-torsion-free and 
$\HS{\FFc}(K_\p,T)/H^1(K_\p,T)^\ru$ has finite length,
\item
if $v \mid \infty$ then $$\HS{\FFc}(K_v,T) := H^1(K_v,T).$$
\end{itemize}
In other words, $\HS{\FFc}(K,T)$ is the Selmer group of 
classes that (after multiplication by some power of $p$) 
are unramified away from $p$, and  universal norms in the 
unramified $\Zp$-extension above $p$.
\end{defn}

Note that the Selmer structure $\FFc$ satisfies \ref{h.5} by Remark \ref{5.2}.

\begin{lem}
\label{urrel}
If $\p\mid p$ then $\corank_R\HS{\FFc^*}(K_\p,T^*) = \corank_RH^0(K_\p,T^*)$.  
\end{lem}

\begin{proof}
By the Lemma in \cite[\S2.1.1]{bprheights} 
(applied to the unramified $\Zp$-extension of $K_\p$), 
$\HS{\FFc^*}(K_\p,T^*)$ is the maximal divisible submodule of the image of the 
(injective) inflation map
$$
H^1(K_\p^\unr/K_\p,(T^*)^{G_{K_\p^\unr}}) \too H^1(K_\p,T^*).
$$
We have 
$$
H^1(K_\p^\unr/K_\p,(T^*)^{G_{K_\p^\unr}}) \cong 
   (T^*)^{G_{K_\p^\unr}}/(\gamma-1)(T^*)^{G_{K_\p^\unr}}
$$
where $\gamma$ is a topological generator of $\Gal(K_\p^\unr/K_\p)$. 
Thus we have an exact sequence 
\begin{multline*}
0 \too H^0(K_\p,T^*) \too (T^*)^{G_{K_\p^\unr}} \map{\gamma-1} (T^*)^{G_{K_\p^\unr}} \\
   \too H^1(K_\p^\unr/K_\p,(T^*)^{G_{K_\p^\unr}}) \too 0 
\end{multline*}
and the lemma follows.
\end{proof}

\begin{cor}
\label{urrelcor}
If $\p\mid p$ and $H^0(K_\p,T^*)$ has finite length, then $\HS{\FFc}(K_\p,T) = H^1(K_\p,T)$.
\end{cor}

\begin{proof}
By Lemma \ref{urrel} $\HS{\FFc^*}(K_\p,T^*)$ has finite length, so 
$H^1(K_\p,T)/\HS{\FFc}(K_\p,T)$ has finite length.  But by definition 
$H^1(K_\p,T)/\HS{\FFc}(K_\p,T)$ is $R$-torsion-free, so $\HS{\FFc}(K_\p,T) = H^1(K_\p,T)$.
\end{proof}

\begin{thm}
\label{dvrdm}
Suppose $R$ is a discrete valuation ring.
Then 
$$
\cm(T,\FFc,\PP) = \sum_{v \mid \infty} \corank_R(H^0(K_v,T^*)).
$$ 
\end{thm}

\begin{proof}
For every $k > 0$ let $T_k = T/\m^k T$.  
If $f, g$ are functions of $k \in \Z^+$, we will write $f(k) \sim g(k)$ to mean that 
$|f(k)-g(k)|$ is bounded independently of $k$.  By definition of core rank 
(see Definition \ref{crd} and Proposition \ref{stprops}(iii)),
the theorem will follow if we can show that 
\begin{equation}
\label{dvd1}
\length(\HS{\FFc}(K,T_k)) - \length(\HS{\FFc^*}(K,T_k^*)) 
   \sim k\sum_{v \mid \infty} \corank_R(H^0(K_v,T^*)).
\end{equation}
By \cite[Proposition 2.3.5]{kolysys} (which is essentially \cite[Lemma 1.6]{wiles}),  
for every $k \in \Z^+$
\begin{multline}
\label{wiles}
\length(\HS{\FFc}(K,T_k)) - \length(\HS{\FFc^*}(K,T_k^*)) \\
   = \length(H^0(K,T_k)) - \length(H^0(K,T_k^*)) \\
    + \sum_{v \in \Sigma(\FFc)} \
            (\length(H^0(K_v,T_k^*)) - \length(\HS{\FFc^*}(K_v,T_k^*))).
\end{multline}

By hypothesis \ref{h.1}, $H^0(K,T_k) = H^0(K,T_k^*) = 0$.  
If $v \mid \infty$, then 
$$
\length(H^0(K_v,T_k^*)) \sim k \;\corank_R(H^0(K_v,T^*)), \quad 
   \length(\HS{\FFc^*}(K_v,T_k^*)) \sim 0.
$$

Suppose $\l \in \Sigma(\FF)$, $\l \nmid p\infty$.  Let $\II_\l$ denote an inertia group 
above $\l$ in $G_K$.  
By \cite[Lemma 1.3.5]{EulerSystems}, we have
$$
\length(\HS{\FFc^*}(K_\l,T_k^*)) \sim \length((T_k^*)^{\II_\l}/(\Frl-1)(T_k^*)^{\II_\l}).
$$
On the other hand, the exact sequence 
$$
0 \to H^0(K_\l,T_k^*) \to (T_k^*)^{\II_\l} 
    \map{\Frl-1} (T_k^*)^{\II_\l} \to (T_k^*)^{\II_\l}/(\Frl-1)(T_k^*)^{\II_\l} \to 0
$$
shows that 
$$
\length(H^0(K_\l,T_k^*)) = \length((T_k^*)^{\II_\l}/(\Frl-1)(T_k^*)^{\II_\l}).
$$
Thus the term for $v=\l$ in \eqref{wiles} is bounded independent of $k$.

Now suppose $\p \mid p$.  By Lemma \ref{urrel}, 
$\corank_R\HFs(K_\p,T^*) = \corank_RH^0(K_\p,T^*)$.
By definition $\HS{\FFc^*}(K_\p,T_k^*)$ is the inverse image of $\HS{\FFc^*}(K_\p,T^*)$ under the 
natural map $H^1(K_\p,T_k^*) \to H^1(K_\p,T^*)[\m^k]$.  A simple exercise shows that the 
kernel and cokernel of this map have length bounded independent of $k$, so 
we see that 
$$
\length(\HS{\FFc^*}(K_\p,T_k^*)) \sim k \;\corank_R\HS{\FFc^*}(K_\p,T^*) 
   = k \;\corank_RH^0(K_\p,T^*).
$$
Thus the term for $v=\p$ in \eqref{wiles} is bounded independent of $k$.

Combining these calculations proves \eqref{dvd1}, and hence the theorem.
\end{proof}

\subsection{Multiplicative groups}
\label{exgm}

Suppose $K$ is a number field and $\rho$ is a character of $G_K$ of finite order.  
For simplicity we will assume that $p > 2$, $\rho$ is nontrivial, and $\rho$ takes values in $\Zp^\times$.  
(Everything that follows holds more generally, only assuming that $\rho$ has 
order prime to $p$, but we would have to tensor everything with the extension 
$\Zp[\rho]$ where $\rho$ takes its values.)

Let $T := \Zp(1) \otimes \rho^{-1}$, a free $\Zp$-module of rank one with $G_K$ acting via the 
product of $\rho^{-1}$ and the cyclotomic character.  
Let $E$ be the cyclic extension of $K$ cut out by $\rho$, i.e., 
such that $\rho$ factors through an injective homomorphism $\Gal(E/K) \hookto \Zp^\times$.  
Let 
$$
\PP = \{\text{primes $\l$ of $K$} : 
   \text{$\l \nmid p$ and $\rho$ is unramified at $\l$}\}.
$$

A simple exercise in 
Galois cohomology (see for example \cite[\S6.1]{kolysys} or \cite[\S1.6.C]{EulerSystems}) 
shows that
$$
H^1(K,T) \cong (E^\times \otimes \Zp)^\rho
$$
where the superscript $\rho$ means the subgroup on which $\Gal(E/K)$ acts via $\rho$, 
and for every prime $\l$, 
$$
H^1(K_\l,T) \cong (E_\l^\times \otimes \Zp)^\rho
$$
where $E_\l = E \otimes_K K_\l$ is the product of the completions of $E$ above $\l$.  
With these identifications, the unramified Selmer structure of Definition \ref{css} 
is given by
$$
\HS{\FFc}(K_\l,T) := (\O_{E,\l}^\times \otimes \Zp)^\rho
$$
for every $\l$, where $\O_{E,\l}$ is the ring of integers of $E_\l$.  

\begin{prop}
\label{13.1}
Let $\Cl(E)$ denote the ideal class group of $E$.  There are natural isomorphisms
$$
\HS{\FFc}(K,T) \cong (\O_{E}^\times \otimes \Zp)^\rho, 
   \quad \HS{\FFc^*}(K,T^*) \cong \Hom(\Cl(E)^\rho,\Qp/\Zp)
$$
and for every $k \ge 0$ an exact sequence
$$
0 \too (\O_E^\times/(\O_E^\times)^{p^k})^\rho \too \HS{\FFc}(K,T/p^kT) \too \Cl(E)[p^k]^\rho \too 0
$$
and an isomorphism
$$
\HS{\FFc^*}(K,T^*[p^k]) \cong \Hom(\Cl(E)^\rho,\Z/p^k\Z).
$$
\end{prop}

\begin{proof}
See for example \cite[Proposition 6.1.3]{kolysys}.
\end{proof}

Suppose in addition now that $\rho \ne \teich$, and either $\rho^2 \ne \teich$ or $p > 3$, 
where $\teich : G_K \to \Zp^\times$ is the Teichm\"uller character 
giving the action of $G_K$ on $\bmu_p$.
Then conditions \ref{h.1}, \ref{h.3}, and \ref{h.4} of \S\ref{sscr} 
all hold.  By Remark \ref{5.2}, the Selmer structure $\FFc$ satisfies \ref{h.5} as well, 
and condition \ref{h.2} holds with $\tau = 1$ and $L = E$.  
Finally, if there is at least one real place $v$ of $K$ such that $\rho$ 
is trivial on complex conjugation at $v$, then the following corollary shows that 
condition \ref{h.6} holds.

\begin{cor}
\label{eighttwo}
The core rank $\cm(T,\FFc)$ is 
$$
\cm(T) = \dim_{\Fp}(\O_E^\times/(\O_E^\times)^{p})^\rho = \rank_{\Zp}(\O_{E}^\times \otimes \Zp)^\rho 
   = |\{\text{archimedean $v$} : \rho(\sigma_v) = 1\}|
$$
where $\sigma_v \in \Gal(E/K)$ is the complex conjugation at $v$.
\end{cor}

\begin{proof}
The first equality follows from Proposition \ref{13.1} and the definition of core rank, and the second 
because $\rho \ne \teich$.  The third equality is well-known (using that $\rho \ne 1$); see for example 
\cite[Proposition I.3.4]{tate}.  
\end{proof}


Thus if $E/K$ is an extension of totally real fields 
and $\rho \ne 1$, then $\cm(T,\FFc) = [K:\Q]$ by Corollary \ref{eighttwo}, 
and all conditions \ref{h.1} through \ref{h.6} are satisfied.

If $K = \Q$, then $\cm(T) = 1$, and a Kolyvagin system (see \S\ref{ksp}) can be constructed 
from the Euler system of cyclotomic units (see \cite{kolysys}).  

For a general totally real field $K$, if we assume the  
version of Stark's Conjecture described in \cite{rubin-stark}, then the 
so-called ``Rubin-Stark'' elements predicted by that conjecture
can be used to construct both an Euler system and a Stark system (see \S\ref{Ssec}).  
For the details and a thorough discussion of this example, see \cite{gen.darmon}.

\subsection{Abelian varieties}
Suppose $A$ is an abelian variety of dimension $d$ defined over the number field $K$.  
Let 
$$
\PP = \{\text{primes $\l$ of $K$} : \text{$\l \nmid p$ and $A$ has good reduction at $\l$}\}.
$$
Let $T$ be the Tate module $T_p(A) := \varprojlim A[p^k]$.  
Then $T$ is a free $\Zp$-module of rank $2d$ 
with a natural action of $G_K$, and $T^* = \check{A}[p^\infty]$ 
where $\check{A}$ is the dual abelian variety to $A$.

Let $\FF$ be the Selmer structure on $T$ given by 
$
\HF(K_v,T) = H^1(K_v,T)
$
for every $v$.  
Then $\FF$ is the unramified Selmer structure $\FFc$ given by Definition \ref{css}. 
(For $v$ dividing $p$, this follows from the Lemma in \cite[\S2.1.1]{bprheights}, 
and for $v$ not dividing $p$ it follows from the fact that $H^1(K_v,T)$ is finite.)
Further, $\FF$ is the usual Selmer structure attached to an abelian variety, with 
the local conditions at primes above $p$ relaxed (see 
for example \cite[\S 1.6.4]{EulerSystems}).  Hence we have 
an exact sequence
$$
0 \too \HF(K,T^*) \too \Sel_{p^\infty}(\check{A}/K) 
   \too \oplus_{\p \mid p} H^1(K_\p,\check{A}[p^\infty]).
$$

Suppose now that $p > 3$, and that the image of $G_K$ in $\Aut(A[p]) \cong \GL_{2d}(\Fp)$ 
is large enough so that conditions \ref{h.1}, \ref{h.2}, and \ref{h.3} of \S\ref{sscr} 
all hold.  For example, this will be true if the image of $G_K$ contains $\GSp_{2d}(\Fp)$.
Condition \ref{h.4} holds since $p>3$, and $\FF$ satisfies \ref{h.5} by Remark \ref{5.2}.
The following consequence of Theorem \ref{dvrdm} shows that condition \ref{h.6} 
holds as well.

\begin{prop}
\label{13.3}
The core rank of $T$ is given by
$
\cm(T) = d\;[K:\Q] .
$
\end{prop}

\begin{proof}
By Theorem \ref{dvrdm}, we have 
$$
\cm(T) = \sum_{v \mid \infty} \corank_{\Zp} H^0(K_v,\check{A}[p^\infty]).
$$  
If $v$ is a real place, then $\corank_{\Zp} H^0(K_v,\check A[p^\infty]) = d$, and if $v$ is a complex place then 
$\corank_{\Zp} H^0(K_v,\check A[p^\infty]) = \corank_{\Zp} \check A[p^\infty] = 2d$.
Thus 
$$
\sum_{v \mid \infty} \corank_{\Zp} H^0(K_v,\check A[p^\infty]) = \sum_{v \mid \infty} d\;[K_v:\R] = d\;[K:\Q].
$$
\end{proof}

If $K = \Q$ and $d = 1$ (i.e., $A$ is an elliptic curve), 
then Proposition \ref{13.3} shows that $\cm(T) = 1$.  In this case Kato has 
constructed an Euler system for $T$, from which one can produce a Kolyvagin system  
(\cite[Theorem 3.2.4]{kolysys}).

\part{Stark systems and the structure of Selmer groups}
\label{part2}

\section{Stark systems}
\label{Ssec}

Suppose for this section that $R$ is a principal artinian ring of length $k$, 
so $\m^k = 0$ and $\m^{k-1} \ne 0$. 
Fix Selmer data $(T,\FF,\PP,r)$ as in Definition \ref{sddef}.  
We assume throughout this section 
that \ref{h.7} of \S\ref{sscr} holds, i.e., $\Il = 0$ for every $\l\in\PP$.

Recall that $\nu(\n)$ denotes the number of prime factors of $\n$.

\begin{defn}
\label{yndef}
For every $\n \in \NN$, define 
\begin{align*}
\W{\n} &:= \oplus_{\l\mid\n} \Hom(\Ht(K_{\l},T),R), \\
\Y{\n} &:= \wedge^{r+\nu(\n)}\HS{\FF^{\n}}(K,T) \otimes \wedge^{\nu(\n)}\W{\n},
\end{align*}
where as usual the exterior powers are taken in the category of $R$-modules.
\end{defn}

Then $\W{\n}$ is a free $R$-module of rank $\nu(\n)$, since each 
$\Ht(K_{\l},T)$ is free of rank one (Lemma \ref{applemma}).  
If we fix an ordering $\n = \l_i \cdots \l_{\nu(\n)}$ of the primes 
dividing $\n$, and a generator $h_i$ of $\Hom(\Ht(K_{\l_i},T),R)$ for every $i$, 
then $h_1\wedge\cdots\wedge h_{\nu(\n)}$ is a generator of the free, 
rank-one $R$-module $\wedge^{\nu(\n)}\W{\n}$.

For the structure of $\Y{\n}$ when $r$ is the core rank of $T$, see Lemma \ref{ylem} below.

\begin{defn}
\label{piqdef}
For every $\l \in \PP$, define the transverse localization map
$$
\loc_{\l}^{\tr} : H^1(K,T) \map{\loc_\l} H^1(K_\l,T) \onto \Ht(K_\l,T),
$$
where the second map is projection (using the direct sum decomposition 
of Lemma \ref{isofunct}(i)) with kernel $\Hf(K_\l,T)$.
If $\n \in \NN$ and $\l\mid\n$, then 
\begin{equation}
\label{gi}
\ker \bigl(\loc_{\l}^{\tr} \bigl| \HS{\FF^{\n}}(K,T) \bigr) = \HS{\FF^{\n/\l}}(K,T).
\end{equation}
In exactly the same way, we can define a map $\loc_\l^{\f}$ 
by using the finite projection and the isomorphism $\fsl$ of 
Definition \ref{fsmap}
$$
\loc_\l^{\f} : H^1(K,T) \map{\loc_\l} H^1(K_\l,T) \onto \Hf(K,T) 
   \map{\fsl} \Ht(K_\l,T) \otimes \Gal(K(\l)_\l/K_\l),
$$ 
and then 
\begin{equation}
\label{gi2}
\ker \bigl(\loc_{\l}^{\f} \bigl| \HS{\FF^{\n}}(K,T) \bigr) = \HS{\FF^{\n/\l}(\l)}(K,T).
\end{equation}
\end{defn}

\begin{defn}
\label{4.3}
Suppose $\n \in \NN$ and $\mm \mid \n$.  By \eqref{gi} we have an exact sequence
$$
0 \too \HS{\FF^\mm}(K,T) \too \HS{\FF^\n}(K,T) 
   \map{\oplus \loc_\l^\tr} \dirsum{\l\mid(\n/\mm)}\Ht(K_\l,T)
$$
and it follows that the square
\begin{equation}
\label{cd}
\raisebox{22pt}
{\xymatrix@R=15pt{
~\HS{\FF^\mm}(K,T)~ \ar@{^(->}[r]\ar_{\oplus \loc_\l^\tr}[d] & ~\HS{\FF^\n}(K,T)~ \ar^{\oplus \loc_\l^\tr}[d] \\
~\dirsum{\l\mid\mm}\Ht(K_\l,T)~ \ar@{^(->}[r] & ~\dirsum{\l\mid\n}\Ht(K_\l,T)~
}}
\end{equation}
is cartesian.  Let
$$
\Psi_{\n,\mm} : \Y{\n} \too \Y{\mm}   
$$
be the map of Proposition \ref{bhprop}(i) attached to this diagram.

Concretely, $\Psi_{\n,\mm}$ is given as follows.
Fix a factorization $\n = \l_1 \cdots \l_t$, with 
$\mm = \l_1 \cdots \l_s$, and a generator $h_i$ of $\Hom(\Ht(K_{\l_i},T),R)$ 
for every $i$.  Let $\n_i = \prod_{j \le i}\l_j$.  These choices lead to a map
$$
\widehat{h_{s+1} \circ \loc_{\l_{s+1}}^{\tr}} \circ \cdots \circ 
   \widehat{h_{t} \circ \loc_{\l_{t}}^{\tr}}
   : \wedge^{r+t}\HS{\FF^{\n}}(K,T) \too \wedge^{r+s} \HS{\FF^{\mm}}(K,T)
$$
(where $\widehat{h_{i} \circ \loc_{\l_{i}}^{\tr}} 
   : \wedge^i\HS{\FF^{\n_i}}(K,T) \to \wedge^{i-1}\HS{\FF^{\n_{i-1}}}(K,T)$
is given by Proposition \ref{wedgemap}) 
and an isomorphism $\wedge^{\nu(\n)}\W{\n} \isom \wedge^{\nu(\mm)}\W{\mm}$ 
given by $h_1\wedge\cdots\wedge h_{t} \mapsto h_1\wedge\cdots\wedge h_{s}$.  
The tensor product of these two maps is the map 
$\Psi_{\n,\mm} : \Y{\n} \too \Y{\mm}$, and is independent of the choices made. 
\end{defn}

\begin{prop}
\label{premagic}
Suppose $\n \in \NN$, $\n' \mid \n$, and $\n'' \mid \n'$.  Then 
$\Psi_{\n',\n''} \circ \Psi_{\n,\n'} = \Psi_{\n,\n''}$.
\end{prop}

\begin{proof}
This is Proposition \ref{bhprop}(iii).
\end{proof}

\begin{defn}
\label{4.6}
Thanks to Proposition \ref{premagic}, we can define the $R$-module 
$\SS_r(T) = \SS_r(T,\FF,\PP)$ of 
{\em Stark systems of rank $r$} to be the inverse limit
$$
\SS_r(T) := \lim_{\overleftarrow{\n\in\NN}} \Y{\n}
$$
with respect to the maps $\Psi_{\n,\mm}$.
\end{defn}

We call these collections Stark systems because a fundamental example 
is given by elements predicted by  a generalized Stark 
conjecture \cite{gen.darmon, rubin-stark}.

Let $\Y{\n}' = \m^{\length(\HS{(\FF^*)_{\n}}(K,T^*))}\Y{\n}$.

\begin{lem}
\label{ylem}
Suppose that hypotheses {\rm\ref{h.1}} through {\rm\ref{h.7}} 
of \S\ref{sscr} are satisfied, so in particular $r$ is the core rank of $T$.  Then:
\begin{enumerate}
\item
$\Y{\n}'$ is a cyclic $R$-module of length $\max\{k - \length(\HS{(\FF^*)_{\n}}(K,T^*)),0\}$.
\item
There are $\n\in\NN$ such that  $\HS{(\FF^*)_{\n}}(K,T^*) = 0$.
\item
If $\HS{(\FF^*)_{\n}}(K,T^*) = 0$ then $\Y{\n}$ is free of rank one over $R$.
\item
If $\HS{(\FF^*)_{\n}}(K,T^*) = 0$ and $\mm\mid\n$, 
then 
$
\Psi_{\n,\mm}(\Y{\n}) = \Y{\mm}'.
$
\end{enumerate}
\end{lem}

\begin{proof}
Assertions (i) and (iii) follow directly from Corollary \ref{threefive}(iv).

Since $\HFs(K,T^*)$ is finite, we can choose generators $c_1, \ldots, c_t$ of $\HFs(K,T^*)[\m]$.
For each $i$, use \cite[Proposition 3.6.1]{kolysys} to choose $\l_i \in \NN$ 
such that $\loc_{\l_i}(c_i) \ne 0$, and let $\n = \prod_i \l_i$. 
Then $\HS{(\FF^*)_{\n}}(K,T^*) = 0$, so (ii) holds.

Proposition \ref{bhprop}(ii) applied to the diagram \eqref{cd} shows that 
$$
\Psi_{\n,\mm}(\Y{\n}) = \m^{\length(\HS{\FF^\mm}(K,T))-(r+\nu(\mm))k}\Y{\mm}.
$$
Corollary \ref{threefive}(ii) shows that 
$$
\length(\HS{\FF^\mm}(K,T))-(r+\nu(\mm))k = \length(\HS{(\FF^*)_{\mm}}(K,T^*))
$$
which proves (iv).
\end{proof}

\begin{thm}
\label{magic}
Suppose that hypotheses {\rm\ref{h.1}} through {\rm\ref{h.7}} 
of \S\ref{sscr} are satisfied.
Then the $R$-module $\SS_r(T)$ is free of rank one, and 
for every $\n\in\NN$, the image of  
the projection map $\SS_r(T) \to \Y{\n}$ is $\Y{\n}'$.
\end{thm}

\begin{proof} 
Using Lemma \ref{ylem}(ii), choose an $\d \in \NN$ such that $\HS{(\FF^*)_{\d}}(K,T^*) = 0$. 
Then $\HS{(\FF^*)_{\n}}(K,T) = 0$ for every $\n \in \NN$ divisible by $\d$.
Now the theorem follows from Lemma \ref{ylem}(iv).
\end{proof}

\section{Stark systems over discrete valuation rings}

For this section we assume that $R$ is a discrete valuation ring, and 
we fix Selmer data $(T,\FF,\PP,r)$ as in Definition \ref{sddef}.
We assume throughout this section that hypotheses \ref{h.1} 
through \ref{h.6} of \S\ref{sscr} are satisfied.  
For $k > 0$ recall from Definition \ref{pndef} that 
$$
\PPk := \{\l \in \PP : \Il \in \m^k\},
$$
and $\NNk$ is the set of squarefree products of primes in $\PPk$.
By Remark \ref{5.7}, the Selmer data $(T/\m^kT,\FF,\PPk,r)$ 
satisfies \ref{h.1} through \ref{h.7} over the ring $R/\m^k$.
In this section we will define the module $\SS_r(T)$ of Stark 
systems of rank $r$ over $T$, and use the results of \S\ref{Ssec} about 
$\SS_r(T/\m^kT)$ to study $\SS_r(T)$.  

\begin{defn}
\label{yndef2}
For every $\n \in \NN$, define 
\begin{align*}
\W{\n} &:= \oplus_{\l\mid\n} \Hom(\Ht(K_{\l},T/\In T),R/\In), \\
\Y{\n} &:= \wedge^{r+\nu(\n)}\HS{\FF^{\n}}(K,T/\In T) \otimes \wedge^{\nu(\n)}\W{\n}, \\
\Y{\n}' &:= \m^{\length(\HS{(\FF^*)_{\n}}(K,T^*[\In]))}\Y{\n}.
\end{align*}
A {\em Stark system of rank $r$} for $T$ (more precisely, for $(T,\FF,\PP)$)
is a collection $\{\epsilon_\n \in \Y{\n} : \n\in\NN\}$ such that if 
$\n\in\NN$ and $\mm \mid \n$, then
$$
\Psi_{\n,\mm}(\epsilon_\n) = \bar\epsilon_\mm
$$
where $\bar\epsilon_\mm$ is the image of $\epsilon_\mm$ in $\Y{\mm} \otimes R/\In$, 
and $\Psi_{\n,\mm} : \Y{\n} \to \Y{\mm} \otimes R/\In$ is the map 
of Definition \ref{4.3} applied to $T/\In T$ and $R/\In$.
Denote by $\SS_r(T) = \SS_r(T,\FF,\PP)$ the $R$-module of Stark systems for $T$.
\end{defn}

\begin{lem}
\label{surjity}
If $j \le k$, then the projection map $T/\m^k T \to T/\m^j T$ and restriction to $\PPk$ 
induce a surjection and an isomorphism, respectively
$$
\xymatrix@C=17pt{
\SS_r(T/\m^k T,\PPk) \ar@{->>}[r] & \SS_r(T/\m^j T,\PPk) & \ar_-{\;\;\sim}[l] \SS_r(T/\m^j T,\PP_j)
}
$$
\end{lem}

\begin{proof}
Let $\n \in \NNk$ be such that 
$\HS{(\FF^*)_{\n}}(K,T^*[\m]) = 0$.  
Then by Theorem \ref{magic}, projecting to $\Y{\n}$ 
gives a commutative diagram with vertical isomorphisms
$$
\xymatrix{
\SS_r(T/\m^k T,\PPk) \ar[r]\ar^{\cong}[d] & \SS_r(T/\m^j T,\PPk) \ar^{\cong}[d]
   & \ar[l]\ar^{\cong}[d] \SS_r(T/\m^j T,\PP_j) \\
\Y{\n} \otimes R/\m^k \ar@{->>}[r] & \Y{\n} \otimes R/\m^j & \Y{\n} \otimes R/\m^j \ar_{=}[l]
}
$$
Since the bottom maps are a surjection and an isomorphism, so are the top ones.
\end{proof}

\begin{prop}
\label{betprp}
The natural maps $T \onto T/\m^k$ and $\PPk \hookto \PP$ induce an isomorphism
$$
\SS_r(T,\PP) \map{\;\sim\;} \varprojlim \SS_r(T/\m^k T,\PPk)
$$
where the inverse limit is with respect to the maps of Lemma \ref{surjity}.
\end{prop}

\begin{proof}
Suppose $\ssys \in \SS_r(T)$ is nonzero.  Then we can find an $\n$ 
such that $\epsilon_\n \ne 0$ in $\Y{\n}$.
If $\n \ne 1$ then $\In \ne 0$, and we let $k$ be such that $\m^k = \In$.  
If $\n = 1$ choose $k$ so that  
$\epsilon_1 \ne 0$ in $\wedge^r\HF(K,T/\m^{k} T)$.
In either case $\In \subset \m^k$, and the image 
of $\ssys$ in $\SS_r(T/\m^k T,\PPk)$ is nonzero.  
Thus the map in the proposition is injective.

Now suppose $\{\ssyi{k}\} \in \varprojlim \SS_r(T/\m^k,\PPk)$.  If $\n \in \NN$ 
and $\n \ne 1$, let $j$ be such that $\In = \m^j$ and  define 
$$
\epsilon_{\n} := \kapi{j}_{\n} \in \Y{\n}.
$$  
If $\n = 1$, define 
$$
\epsilon_1 = \lim_{k\to\infty} \kapi{k}_{1} \in \lim_{k\to\infty}\wedge^r\HF(K,T/\m^kT) 
   = \wedge^r\HF(K,T) = \Y{1}.
$$
It is straightforward to verify that this defines an element $\ssys \in \SS_r(T,\PP)$ 
that maps to $\ssyi{k} \in \SS_r(T/\m^k T,\PPk)$ for every $k$.
Thus the map in the proposition is surjective as well.
\end{proof}

\begin{thm}
\label{rankdvr}
Suppose $R$ is a discrete valuation ring and hypotheses \ref{h.1} through \ref{h.6} hold.
Then the $R$-module of Stark systems of rank $r$, $\SS_r(T,\PP)$, is free of rank one, generated by 
a Stark system $\ssys$ whose image in $\SS_r(T/\m T,\PP)$ is nonzero.  The map 
$\SS_r(T,\PP) \to \SS_r(T/\m^k,\PPk)$ is surjective for every $k$.
\end{thm}

\begin{proof}
By Theorem \ref{magic},   
$\SS_r(T/\m^kT,\PPk)$ is free of rank one over $R/\m^k$ for every $k$.  
The maps $\SS_r(T/\m^{k+1}T,\PP_{k+1}) \to \SS_r(T/\m^kT,\PPk)$ 
are surjective by Lemma \ref{surjity}, so the theorem follows from  
Proposition \ref{betprp}.
\end{proof}

\section{Structure of the dual Selmer group}
\label{extra}

In this section $R$ is either a principal artinian local ring or a discrete 
valuation ring.  We let $k := \length(R)$, so $k$ is finite in the artinian case 
and $k = \infty$ in the discrete valuation ring case.

Fix Selmer data $(T,\FF,\PP,r)$.  
We continue to assume that hypotheses \ref{h.1} 
through \ref{h.6} are satisfied, 
and if $R$ is artinian we assume that \ref{h.7} is satisfied as well.
Recall that 
if $\n\in\NN$ then $\nu(\n)$ denotes the number of prime divisors of $\n$.

\begin{defn}
\label{allde}
Define functions $\smesh, \mesh, \divity_{\ssys} \in \Maps(\NN,\Z_{\ge 0}\cup\{\infty\})$
\begin{itemize}
\item
$\smesh(\n) = \length(\HS{(\FF^*)_{\n}}(K,T^*))$,
\item
$\mesh(\n) = \length(\HFfns(K,T^*))$,
\end{itemize}
and if $\ssys \in \SS_r(T)$ is a Stark system
\begin{itemize}
\item
$\divity_{\ssys}(\n) = \max\{j : \epsilon_\n \in \m^j\Y{\n}\}$.
\end{itemize}
Define $\partial : \Maps(\NN,\Z_{\ge 0}\cup\{\infty\}) \to \Maps(\Z_{\ge 0},\Z_{\ge 0}\cup\{\infty\})$ 
by
$$
\partial f(i) = \min\{f(\n) : \text{$\n \in \NN$ and $\nu(\n) = i$}\}.
$$ 
%
\end{defn}

\begin{defn}
\label{allder}
The {\em order of vanishing} of a nonzero Stark system $\ssys \in \SS_r(T)$ is 
$$
\ordv{\ssys} := \min\{\nu(\n) : \n \in \NN, \epsilon_\n \ne 0\}
   = \min\{i : \partial\divity_{\ssys}(i) \ne \infty\}.
$$
%
We say $\ssys \in \SS_r(T)$ is 
{\em primitive} if its image in $\SS_r(T/\m T)$ is nonzero.
We also define the  sequence of {\em elementary divisors} 
$$
d_{\ssys}(i) := \partial\divity_{\ssys}(i) - \partial\divity_{\ssys}(i+1), \quad i \ge \ordv{\ssys}.
$$
\end{defn}

Note that $\partial\divity_{\ssys}(i) = \infty$ if $i < \ordv{\ssys}$;  
Theorems \ref{struc} and \ref{korank} below show that the converse is true as well, 
so the $d_{\ssys}(i)$ are well-defined and finite.

\begin{prop}
\label{6.12}
Suppose $R$ is artinian, 
and $\HFs(K,T^*) \cong \oplus_{i\ge 1} R/\m^{e_i}$ with $e_1 \ge e_2 \ge \cdots$.  
Then for every $t \ge 0$, 
$$
\partial\mesh(t) = \partial\smesh(t) = \sum_{i > t} e_i.
$$
\end{prop}

\begin{proof}
Suppose $\n \in \NN$ and $\nu(\n) = t$.  Consider the map
$$
\HFs(K,T^*) \too \dirsum{\l \mid \n} \Hf(K_\l,T^*).
$$
The right-hand side is free of rank $t$ over $R$, and $R$ is principal, so the image is 
a quotient of $\HFs(K,T^*)$ generated by (at most) $t$ elements.  
Hence the image has length at most $\sum_{i \le t} e_i$, so 
the kernel has length at least $\sum_{i>t} e_i$.  
But by definition this kernel is $\HS{(\FF^*)_{\n}}(K,T^*)$, which is contained in 
$\HFfns(K,T^*)$, so
\begin{equation}
\label{eq:3}
\mesh(\n) \ge \smesh(\n) \ge \sum_{i>t} e_i.
\end{equation}

We will prove by induction on $t$ that $\n$ can be chosen so that 
$\nu(\n) = t$ and $\HFfns(K,T^*) \cong \oplus_{i>t} R/\m^{e_i}$.  For such 
an $\n$ equality holds in \eqref{eq:3}, and the lemma follows.  
When $t = 0$ we can just take $\n = 1$.

Suppose we have an $\n$ with $\nu(\n) = t-1$ and 
$\HFfns(K,T^*) \cong \oplus_{i>t-1} R/\m^{e_i}$.
Since $\cm(T) > 0$, Corollary \ref{threefive} shows that 
$\m^{k-1}\HFfn(K,T) \ne 0$.
Fix a nonzero element $c \in \m^{k-1}\HFfn(K,T) \subset \HFfn(K,T)[\m]$.
If $e_t > 0$ then choose a nonzero element 
$c' \in \m^{e_t-1}\HFfns(K,T^*) \subset \HFfns(K,T^*)[\m]$.  
By \cite[Proposition 3.6.1]{kolysys} we can use the Cebotarev theorem 
to choose a prime $\l \in \PP$ such that 
the localization $\loc_\l(c) \ne 0$ and, if $e_t > 0$, such that 
$\loc_\l(c') \ne 0$ as well.

Since $\Hf(K_\l,T)$ is free of rank one over $R$, and (by our choice of $\l$)
the localization of $\m^{k-1}\HFfn(K,T)$ at $\l$ is nonzero, it follows that 
the localization map $\HFfn(K,T) \to \Hf(K_\l,T)$ is surjective.
Similarly, we have that $\HS{\FF(\n)^*}(K,T^*)$ has exponent $\m^{e_t}$, and 
if $e_t > 0$ then the localization of $\m^{e_t-1}\HFfns(K,T^*)$ at $\l$ 
is nonzero, so 
$$
\HS{\FF(\n)^*}(K,T^*)/\HS{\FF^\l(\n)^*}(K,T^*) \cong \loc_\l(\HS{\FF(\n)^*}(K,T^*)) \cong R/\m^{e_t}
$$
and therefore $\HS{\FF^\l(\n)^*}(K,T^*) \cong \oplus_{i > t} R/\m^{e_i}$.
By \cite[Theorem 4.1.7(ii)]{kolysys} we have 
$\HFfnls(K,T^*) = \HS{\FF^\l(\n)^*}(K,T^*)$, so $\n\l \in \NN$ has the desired property.
\end{proof}

\begin{prop}
\label{struclem}
Suppose $R$ is artinian of length $k$, and $\ssys \in \SS_r(T)$.  
Fix $s \ge 0$ such that $\ssys$ generates $\m^s \SS_r(T)$,
and nonnegative integers $e_1 \ge e_2 \ge \cdots$ such that
$$
\HFs(K,T^*) \cong \oplus_i R/\m^{e_i}.
$$ 
Then for every $t \ge 0$, 
$$
\partial\divity_{\ssys}(t) = 
\begin{cases}
s+\sum_{i>t} e_i & \text{if $s+\sum_{i>t} e_i < k$}, \\
\infty & \text{if $s+\sum_{i>t} e_i \ge k$}.
\end{cases}
$$
\end{prop}

\begin{proof}
It is enough to prove the proposition when $s = 0$, and the general case will follow.
So we may assume that $\ssys$ generates $\SS_r(T)$.  By Theorem \ref{magic} 
and Lemma \ref{ylem}(i), 
we have that $\epsilon_\n$ generates $\Y{\n}' = \m^{\smesh(\n)}\Y{\n}$, which is cyclic 
of length $\max\{k-\smesh(\n),0\}$.
Hence $\epsilon_\n \in \m^{\smesh(\n)}\Y{\n}$, and $\epsilon_\n \in \m^{\smesh(\n)+1}\Y{\n}$ 
if and only if $\smesh(\n) \ge k$.
Therefore
$$
\partial\divity_{\ssys}(t) = 
\begin{cases}
\partial\smesh(t) & \text{if $\partial\smesh(t) < k$}, \\
\infty & \text{if $\partial\smesh(t) \ge k$}.
\end{cases}
$$
Now the proposition follows from the calculation of $\partial\smesh(t)$ in Lemma \ref{6.12}.
\end{proof}

\begin{thm}
\label{struc}
Suppose $R$ is artinian, $\ssys \in \SS_r(T)$, and $\epsilon_1 \ne 0$.  Then 
$$
\renewcommand{\arraystretch}{1.4}
\arraycolsep=2pt
\begin{array}{ccccccccc}
\partial\divity_{\ssys}(0) & \ge & \partial\divity_{\ssys}(1) 
   & \ge & \partial\divity_{\ssys}(2) & \ge & \cdots, \\
d_{\ssys}(0) & \ge & d_{\ssys}(1) & \ge & d_{\ssys}(2) & \ge & \cdots & \ge & 0,
\end{array}
$$
and
$$
\HFs(K,T^*) \cong \dirsum{i \ge 0} R/\m^{d_{\ssys}(i)}.
$$
\end{thm}

\begin{proof}
Let $s$ be such that $\ssys$ generates $\m^s \SS_r(T)$.  
If $\epsilon_1 \ne 0$ then $\partial\divity_{\ssys}(0) < k$, so in 
Proposition \ref{struclem} we have $\partial\divity_{\ssys}(t) = s + \sum_{i>t} e_i$ for every $t$.
The theorem follows directly.
\end{proof}

If $R$ is a discrete valuation ring then $F$ will denote 
the field of fractions of $R$, and if $M$ is an $R$-module we define 
\begin{itemize}
\item 
$\rank_R M := \dim_{F}M \otimes F$,
\item
$\corank_R M := \rank_R \Hom_R(M,F/R)$,
\item
$M_\dv$ is the maximal divisible submodule of $M$.
\end{itemize}

\begin{prop}
\label{11.4}
Suppose $R$ is a discrete valuation ring, and $\ssys \in\SS_r(T)$ generates $\m^s\SS_r(T)$.  
Let $a := \corank_R(\HFs(K,T^*))$ and write
$$
\HFs(K,T^*)/(\HFs(K,T^*))_\dv \cong \dirsum{i>a} R/\m^{e_i}
$$
with $e_{a+1} \ge e_{a+2} \ge \cdots$. 
Then
$$
\partial\divity_{\ssys}(t) = 
\begin{cases}
\infty & \text{if $t < a$}, \\
s+\sum_{i>t} e_i & \text{if $t \ge a$}.
\end{cases}
$$
\end{prop}

\begin{proof}
Let $e_1 = \cdots = e_a := \infty$.  
Since 
$$
\HFs(K,T^*) = \varinjlim \HFs(K,T^*[\m^k]),
$$ 
Proposition \ref{stprops}(ii) applied to all the $T/\m^kT$ shows that 
for every $k \in \Z^+$ we have 
\begin{equation}
\label{eq:n}
\HFs(K,T^*[\m^k]) = \HFs(K,T^*)[\m^k] \cong 
    \dirsum{i \ge 1} R/\m^{\min\{k,e_i\}}.
\end{equation}
For every $k \ge 0$ let $\ssyi{k}$ denote the image of $\ssys$ in $\SS_r(T/\m^k T,\PPk)$.
Fix $s \ge 0$ such that $\ssys$ generates $\m^s\SS_r(T)$.  Then by Theorem \ref{rankdvr}, 
$\ssyi{k}$ generates $\m^s\SS_r(T/\m^k T)$ for every $k$.

Fix $t$, and choose $\n \in \NN$ with $\nu(\n) = t$.  Let $k$ be such that $\In = \m^k$.  
By \eqref{eq:n} and Proposition \ref{struclem} we have that 
$\epsilon^{(k)}_\n = 0$ if $t < a$, and $\epsilon^{(k)}_\n \in \m^{s+\sum_{i>t}e_i}\Y{\n}$ if $t > a$.  
But $\epsilon^{(k)}_\n = \epsilon_\n \in \Y{\n}$, so we conclude that 
$\partial\divity_{\ssys}(t) = \infty$ if $t < a$, and 
$\partial\divity_{\ssys}(t) \ge s+\sum_{i>t}e_i$ if $t \ge a$.

Now suppose $t \ge a$, and fix $k > s+\sum_{i>t}e_i$.  By Proposition \ref{struclem} we can 
find $\n\in \NN$ with $\In \subset \m^k$ such that 
$\epsilon^{(k)}_\n \notin \m^{s+1+\sum_{i>t}e_i}\Y{\n}$.
Since $\epsilon^{(k)}_\n$ is the image of $\epsilon_\n$, we have that 
$\epsilon_\n \notin \m^{s+1+\sum_{i>t}e_i}\Y{\n}$.
This shows that $\partial\divity_{\ssys}(t) \le s+\sum_{i>t}e_i$, and the proof is complete.
\end{proof}

\begin{thm}
\label{korank}
Suppose $R$ is a discrete valuation ring, $\ssys \in \SS_r(T)$ and $\ssys \ne 0$.  Then:
\begin{enumerate}
\item
the sequence $\partial\divity_{\ssys}(t)$ is nonincreasing, finite for $t \ge \ordv{\ssys}$, 
and nonnegative,
\item
the sequence $d_{\ssys}(i)$ is nonincreasing, finite for $i \ge \ordv{\ssys}$, 
and nonnegative,
\item
$\ordv{\ssys}$ and the $d_{\ssys}(i)$ are independent of the choice of 
nonzero $\ssys \in \SS_r(T)$, 
\item
$\corank_R(\HFs(K,T^*)) = \ordv{\ssys}$, 
\item
$\HFs(K,T^*)/(\HFs(K,T^*))_\dv \cong 
    \oplus_{i\ge\ordv{\ssys}} R/\m^{d_{\ssys}(i)}$,
\item
$\length_R(\HFs(K,T^*)/(\HFs(K,T^*))_\dv) 
    = \partial\divity_{\ssys}(\ordv{\ssys})-\partial\divity_{\ssys}(\infty)$, 
where $\partial\divity_{\ssys}(\infty) := \lim_{t \to \infty}\partial\divity_{\ssys}(t)$
\item
$\ssys$ is primitive if and only if $\der{\infty}(\ssys) = 0$,
\item
$\length(\HFs(K,T^*))$ is finite if and only if $\epsilon_1 \ne 0$,
\item
$
\length(\HFs(K,T^*)) \le \partial\divity_{\ssys}(0) = \max\{s : \epsilon_1 \in \m^s\wedge^r\HF(K,T)\},
$
with equality if and only if $\ssys$ is primitive.
\end{enumerate}
\end{thm}

\begin{proof}
The theorem follows directly from Proposition \ref{11.4}.
\end{proof}

\part{Kolyvagin systems}
\label{part3}

\section{Sheaves and monodromy}
\label{sog}

In this section we recall some concepts and definitions from \cite{kolysys}.

\begin{defn}
\label{sheafdef}
If $X$ is a graph, a sheaf $\cS$ (of $R$-modules) on $X$ is a rule
assigning:
\begin{itemize}
\item 
to each vertex $v$ of $X$, an $R$-module $\cS(v)$ (the stalk of $X$ at $v$),
\item 
to each edge $e$ of $X$, an $R$-module $\cS(e)$,
\item  
to each pair $(e, v)$ where $v$ is an endpoint of the edge $e$, 
an $R$-module map $\pve : \cS(v) \to \cS(e)$.
\end{itemize}
A global section of $\cS$ is a collection $\{\kappa_v \in \cS(v) : v \in V\}$ 
such that for every edge $e \in E$, if $e$ has endpoints $v, v'$ then 
$\pve(\kappa_v) = \pvpe(\kappa_{v'})$ in $\cS(e)$.
We write $\Gamma(\cS)$ for the $R$-module of global sections of $\cS$.
\end{defn}

\begin{defn}
\label{lcdef}
We say that a sheaf $\cS$ on a graph $X$ is
{\em locally cyclic} if all the $R$-modules $\cS(v)$, $\cS(e)$
are cyclic and all the maps $\pve$ are surjective. 

If $\cS$ is locally cyclic then a {\em surjective path} (relative to $\cS$)
from $v$ to $w$ is a path $(v=v_1,v_2,\ldots,v_k=w)$
in $X$ such that for each $i$, if $e_i$ is the edge joining $v_i$ and 
$v_{i+1}$, then $\shmap{v_{i+1}}{e_i}$ is an isomorphism.
We say that the vertex $v$ is a {\em hub} of $\cS$
if for every vertex $w$ there is an $\cS$-surjective path from $v$ to $w$.

Suppose now that the sheaf $\cS$ is locally cyclic.  If
$P = (v_1,v_2,\ldots,v_k)$ is a surjective path in $X$, we can define a
surjective map $\shmap{P}{} : \cS(v_1) \to \cS(v_k)$ by
$$
\shmap{P}{} :=
(\shmap{v_k}{e_{k-1}})^{-1} \circ \shmap{v_{k-1}}{e_{k-1}}
     \circ (\shmap{v_{k-1}}{e_{k-2}})^{-1}
     \circ \cdots \circ (\shmap{v_2}{e_1})^{-1} \circ \shmap{v_1}{e_1}
$$
since all the inverted maps are isomorphisms.
We will say that $\cS$ has {\em trivial monodromy} if 
whenever $v, w, w'$ are vertices, $P, P'$ are surjective paths  
$(v,\ldots,w)$ and $(v,\ldots,w')$, and $w, w'$ are joined by an edge $e$, 
then 
$\shmap{w}{e}\circ\shmap{P}{} = \shmap{w'}{e}\circ\shmap{P'}{} 
    \in \Hom(\cS(v),\cS(e))$.
In particular 
for every pair $v, w$ of vertices and
and every pair $P, P'$ of surjective paths from $v$ to $w$, we require that
$\shmap{P}{} = \shmap{P'}{} \in \Hom(\cS(v),\cS(w))$.
\end{defn}

\begin{prop}
\label{lcgraphprop}
Suppose $\cS$ is locally cyclic and $v$ is a hub of $\cS$.
\begin{enumerate}
\item
The map $f_v : \Gamma(\cS) \to \cS(v)$ defined by $\ksys \mapsto \kappa_v$
is injective, and is surjective if and only if $\cS$ has trivial monodromy.
\item
If $\ksys \in \Gamma(\cS)$, and if $u$ is a vertex such that
$\kappa_u \ne 0$ and $\kappa_u$ generates $\m^i \cS(u)$ for some
$i \in \Z^+$, then $\kappa_w$ generates $\m^i \cS(w)$ for every vertex $w$.
\end{enumerate}
\end{prop}

\begin{proof}
This is \cite[Proposition 3.4.4]{kolysys}.
\end{proof}

\begin{defn}
\label{perf}
A  global section $\kappa \in \Gamma(\cS)$ will be called {\em primitive} 
if for every vertex $v$, $\kappa(v) \in \cS(v)$ is a generator of the 
$R$-module $\cS(v)$.  

It follows from Proposition \ref{lcgraphprop} that a locally cyclic 
sheaf $\cS$ with a hub has a primitive global section if and only if $\cS$ has 
trivial monodromy.
\end{defn}

\section{Kolyvagin systems and the Selmer sheaf}
\label{ksp}

Fix Selmer data $(T,\FF,\PP,r)$ as in Definition \ref{sddef}.  
Recall that we have defined a Selmer structure $\FF(\n)$ for every $\n\in\NN$ 
(Definition \ref{FFabc}) by modifying the local condition at primes dividing $\n$, 
and that $K(\l)$ is the $p$-part of the ray class field of $K$ modulo $\l$.

\begin{defn}
For every $\n \in\NN$, define 
$$
\Gn := \tensor{\l \mid \n} \Gal(K(\l)_\l/\Kl).
$$
Each $\Gal(K(\l)_\l/\Kl)$ is cyclic with order contained in $\In$, so $\Gn \otimes (R/\In)$ is free of rank one over $R/\In$.

If $\l$ is a prime dividing $\n$, then $(T/\In T)/(\Frl-1)(T/\In T)$ is free of rank one over 
$R/\In$, so we can apply the results of \S\ref{lcg} to $H^1(\Kl,T/\In T)$.  In particular we will write 
$$
\fsl : \Hf(\Kl,T/\In T) \too \Ht(\Kl,T/\In T) \otimes G_\l
$$ 
for the finite-singular isomorphism of Definition \ref{fsmap} applied to $\Kl$.

If $\l$ is a prime, $\n\l \in \NN$, and $r \ge 1$, then we can 
compare $\wedge^r \HFfn(K,T/\In T) \otimes \Gn$ and 
$\wedge^r \HS{\FFf{\n\l}}(K,T/\Inl T) \otimes \Gnl$ using the exterior algebra of Appendix \ref{exalg}.
Namely, applying Proposition \ref{wedgemap} with the localization maps of Definition \ref{piqdef}
\begin{align*}
&\loc_\l^\f : \HFfn(K,T/\Inl T) \too \Hf(K_\l,T/\Inl T) \map{\fsl} \Ht(K_\l,T/\Inl T) \otimes G_\l, \\ 
&\loc_\l^\tr : \HS{\FFf{\n\l}}(K,T/\Inl T) \too \Ht(K_\l,T/\Inl T)
\end{align*}
gives the top and bottom maps, respectively, in the following diagram:
\begin{equation}
\label{compdiag}
\hskip -50pt
\raisebox{48pt}
{\xymatrix@R=25pt@C=15pt{
(\wedge^r\HFfn(K,T/\In T)) \otimes \Gn \ar^(.6){\widehat{\loc_\l^f}\otimes 1}[dr] \\
    & \hskip -80pt \Ht(\Kl,T/\Inl T) \otimes (\wedge^{r-1}\HS{\FF_\l(\n)}(K,T/\Inl T)) 
       \otimes \Gnl \hskip -80pt \\
    (\wedge^r\HS{\FFf{\n\l}}(K,T/\Inl T)) \otimes \Gnl  \ar_(.6){\widehat{\loc_\l^\tr}\otimes 1}[ur]
}}
\end{equation}
\end{defn}

\begin{defn}
\label{canonicalsheaf}
Define a graph $\X := \X(\PP)$ by taking the set of vertices of $\X$ to be $\NN := \NN(\PP)$ 
(Definition \ref{pndef}), and whenever 
$\n, \n\l \in \NN$ (with $\l$ prime) we join $\n$ and $\n\l$ by an edge.

The {\em Selmer sheaf} associated to $(T,\FF,\PP,r)$ is the
sheaf $\HH = \HH_{(T,\FF,\PP,r)}$ of $R$-modules on $\X$ defined as follows.  
Let 
\begin{itemize}
\item
$\HH(\n) := (\wedge^r\HFfn(K,T/\In T)) \otimes \Gn$ for $\n \in \NN$, 
\end{itemize}
and if $e$ is the edge joining $\n$ and $\n\l$ define 
\begin{itemize}
\item
$\HH(e) := \Ht(\Kl,T/\Inl T) \otimes (\wedge^{r-1}\HS{\FF_\l(\n)}(K,T/\Inl T)) \otimes \Gnl$,
\item
$\pne : \HH(\n) \to \HH(e)$ is the upper map of \eqref{compdiag},
\item
$\pnle : \HH(\n\l) \to \HH(e)$ is the lower map of \eqref{compdiag}.
\end{itemize}
We call $\HH(n) := \wedge^r\HFfn(K,T/\In T) \otimes \Gn$ the {\em Selmer stalk} at $n$.
\end{defn}

\begin{defn}
\label{ksdefn}
A {\em Kolyvagin system} for $(T,\FF,\PP,r)$ 
(or simply a Kolyvagin system of rank $r$ for $T$, if $\FF$ and $\PP$ are fixed) 
is a global section of the Selmer sheaf $\HH$.
We write $\KS_r(T,\FF,\PP)$, or simply $\KS_r(T)$ when there is no 
risk of confusion, for the $R$-module of Kolyvagin systems $\Gamma(\HH)$.  

Concretely, a Kolyvagin system for $(T,\FF,\PP,r)$ is a collection of classes 
$$
\{\kappa_\n \in (\wedge^r\HFfn(K,T/\In T)) \otimes \Gn : \n \in \NN\}
$$
such that if $\l$ is prime and $\n\l \in \NN$,
the images of $\kappa_{\n}$ and $\kappa_{\n\l}$ coincide in the diagram
\eqref{compdiag}.
\end{defn}

\begin{rem}
The definition of Kolyvagin system given in \cite{kolysys} corresponds to 
the definition above with $r = 1$.
\end{rem}

\section{Stub Kolyvagin systems}
\label{stubks}

Suppose until the final result of this section that 
$R$ is a principal artinian ring of length $k$.
Fix Selmer data $(T,\FF,\PP,r)$ as in Definition \ref{sddef}
such that hypotheses \ref{h.1} through \ref{h.7} of \S\ref{sscr} hold. 
In particular $r = \cm(T)$ is the core rank of $T$.

Recall that for $\n \in \NN$ we defined
$$
\mesh(\n) := \length_R(\HFfns(K,T^*)) \in \Z_{\ge 0} \cup \{\infty\}.
$$
We say that a vertex $\n \in \NN$ is a {\em core vertex} if $\mesh(\n) = 0$.

\begin{prop}
The following are equivalent:
\begin{enumerate}
\item
$\n$ is a core vertex for $T$,
\item
$\HFfn(K,T)$ is free of rank $r$ over $R$, 
\item
$\HH(\n)$ is free of rank one over $R$,
\item 
$\n$ is a core vertex for $T/\m T$.
\end{enumerate}
\end{prop}

\begin{proof}
We have (i) $\iff$ (ii) by Corollary \ref{threefive}, and 
(i) $\iff$ (iv) by Proposition \ref{stprops}(ii).  It is easy to see that 
(ii) $\iff$ (iii).
\end{proof}

\begin{prop}
\label{eqim}
If $\n,\n\l \in \NN$ and $e$ is the edge joining them, then
$$
\pne(\m^{\mesh(\n)} \HH(\n)) = \pnle(\m^{\mesh(\n\l)} \HH(\n\l)) \subset \HH(e).
$$
\end{prop}

\begin{proof}
By Proposition \ref{wedgemap}(ii) and Definition \ref{canonicalsheaf} of $\pne$ and $\pnle$, we have 
\begin{align*}
\pne(\HH(\n)) &= \fsl(\loc_\l(\HFfn(K,T))) \otimes \wedge^{r-1}\HS{\FF_\l(\n)}(K,T) \otimes \Gn, \\
\pnle(\HH(\n\l)) &= \loc_\l(\HFfnl(K,T)) \otimes \wedge^{r-1}\HS{\FF_\l(\n)}(K,T) \otimes \Gnl.
\end{align*}
By \cite[Lemma 4.1.7]{kolysys}, global duality shows that
$$
\m^{\mesh(\n)}\fsl(\loc_\l(\HFfn(K,T))) = \m^{\mesh(\n\l)}\loc_\l(\HFfnl(K,T)) \otimes \Gl
$$
and the proposition follows.
\end{proof}

We define a subsheaf $\HH'$ of the Selmer sheaf $\HH$ as follows.

\begin{defn}
\label{stubsheaf}
The {\em sheaf of stub Selmer modules} $\HH' = \HH'_{(T,\FF,\PP,r)} \subset \HH$ 
is the subsheaf of $\HH$ defined by
\begin{itemize}
\item
$
\HH'(\n) := \m^{\mesh(\n)} \HH(\n) 
    = \m^{\mesh(\n)} (\wedge^r\HFfn(K,T)) \otimes \Gn \subset \HH(\n)
$
if $\n \in \NN$,
\item
$\HH'(e)$ is the image of $\HH'(\n)$ in $\HH(e)$ 
under the vertex-to-edge map of $\HH$, 
if $\n$ is a vertex of the edge $e$ (this is well-defined by Proposition \ref{eqim}),
\end{itemize}
and the vertex-to-edge maps are the restrictions of those of the sheaf $\HH$.
\end{defn}

%

\begin{defn}
\label{ks'def}
A {\em stub Kolyvagin system} is a global section of the sheaf $\HH'$.  
We let $\KS'_r(T) = \KS_r'(T,\FF,\PP) := \Gamma(\HH') \subset \KS_r(T)$ 
denote the $R$-module of stub Kolyvagin systems.
\end{defn}

\begin{rem}
It is shown in \cite[Theorem 4.4.1]{kolysys} that when the core rank $\cm(T) = 1$, 
we have $\KS'_1(T) = \KS_1(T)$.  In other words, in that case 
for every Kolyvagin system $\ksys \in \KS_1(T)$ and $\n \in \NN$, we have 
$
\kappa_\n \in \m^{\mesh(\n)}\HFfn(K,T) \otimes \Gn.
$

\end{rem}

\begin{thm}
\label{mthm0}
\begin{enumerate}
\item
There are core vertices.
\item
Suppose $\n,\n'$ are core vertices.  Then there is a path
$$
\xymatrix{
\n=\n_0 \ar@{-}^-{e_1}[r] & \n_1 \ar@{-}^-{e_2}[r] & \cdots \ar@{-}^-{e_t}[r] & \n_t=\n'
}
$$
in $\X$ such that every $\n_i$ is a core vertex and 
all of the maps $\psi_{\n_i}^{e_{i+1}}$ and $\psi_{\n_i}^{e_{i}}$ 
are isomorphisms.
\item
The stub subsheaf $\HH'$ is locally cyclic, and every core vertex is a hub.  
For every vertex $\n \in \NN$, there is a 
core vertex $\n' \in \NN$ divisible by $\n$.
\end{enumerate}
\end{thm}

Theorem \ref{mthm0} will be proved in \S\ref{m0pf}.  In the remainder of this 
section we derive some consequences of it.

\begin{thm}
\label{mthm3}
\begin{enumerate}
\item
The module $\KS'_r(T)$ of stub Kolyvagin systems is free of rank one over $R$, and 
for every core vertex $\n$ the specialization map 
$$
\KS'_r(T) \too \HH'(\n) = (\wedge^r\HFfn(K,T)) \otimes \Gn
$$ 
given by $\ksys \mapsto \kappa_\n$ is an isomorphism.
\item
There is a Kolyvagin system $\ksys \in \KS'_r(T)$ such that $\kappa_\n$ generates $\HH'(\n)$ 
for every $\n \in \NN$.
\item
The locally cyclic sheaf $\HH'$ has trivial monodromy.  
\end{enumerate}
\end{thm}

\begin{proof}
This follows from Proposition \ref{lcgraphprop}, using Theorem \ref{mthm0}(i,iii).
\end{proof}

For the next theorem we take $R$ to be a discrete valuation ring.

\begin{thm}
\label{betthm}
Suppose that $R$ is a discrete valuation ring, and hypotheses {\rm\ref{h.1}} 
through {\rm\ref{h.6}} are satisfied for the Selmer data $(T,\FF,\PP,r)$.  
For $k > 0$ let $\PPk \subset \PP$ be as in Definition \ref{pndef}.

The natural maps $T \onto T/\m^k$ and $\PPk \hookto \PP$ induce an isomorphism
$$
\KS'_r(T,\PP) \map{\;\sim\;} \varprojlim \KS'_r(T/\m^k T,\PPk).
$$
The $R$-module $\KS'_r(T,\PP)$, is free of rank one, generated by 
a Kolyvagin system $\ksys$ whose image in $\KS'_r(T/\m T)$ is nonzero.  The maps 
$\KS'_r(T,\PP) \to \KS'_r(T/\m^k,\PPk)$ are surjective.
\end{thm}

\begin{proof}
This can be proved easily directly from Theorem \ref{mthm3}, as in the 
proofs of Proposition \ref{betprp} and Theorem \ref{rankdvr} for Stark systems.
See also \cite[Proposition 5.2.9]{kolysys}.
\end{proof}

\begin{rem}
When $r = \cm(T) > 1$, it is not generally true that $\KS'_r(T) = \KS_r(T)$.  
For example, suppose $R$ is principal artinian of length $k > 1$, and 
suppose $\mm\in\NN$ is such that $\HS{\FF(\mm)}(K,T) \cong R^r \oplus (R/\m)^r$, 
with corresponding basis $c_1,\ldots,c_r, d_1,\ldots,d_r$.  Let $g_\mm$ 
be a generator of $G_\mm$.

For every $\l\in\PP$ and every $i$, $\loc_\l(d_i)$ is killed by $\m$, 
so it is divisible by $\m^{k-1}$ in the free $R$-module $H^1(K_\l,T)$.  
It follows that if we define $\ksys := \{\kappa_\n\}$ where
$$
\kappa_\n := 
\begin{cases}
(d_1 \wedge \cdots \wedge d_r) \otimes g_\mm & \text{if $\n=\mm$},\\
0 & \text{if $\n\ne\mm$},
\end{cases}
$$
then $\ksys$ is a Kolyvagin system, but $\kappa_\mm \notin \HH'(\mm)$ so $\ksys \notin \KS_r'(T)$.
\end{rem}

\section{Kolyvagin systems and Stark systems}
\label{ksnss}

Suppose that $R$ is a principal artinian ring, and 
fix Selmer data $(T,\FF,\PP,r)$ as in Definition \ref{sddef} such that 
$\Il = 0$ for every $\l \in \PP$.
Recall the $R$ module $\Y{\n}$ of Definition \ref{yndef}, and let 
$\loc^\f_\l : H^1(K,T) \to \Ht(K,T) \otimes \Gl$ and 
$\loc^\tr_\l : H^1(K,T) \to \Ht(K,T)$ be the maps of Definition \ref{piqdef}.

\begin{defn}
Suppose $\n\in\NN$.  By \eqref{gi2} we have an exact sequence
$$
0 \too \HFfn(K,T) \too \HS{\FF^\n}(K,T) 
   \map{\oplus \loc_\l^\f} \dirsum{\l\mid\n}\Ht(K_\l,T) \otimes \Gl
$$
and it follows that the square
\begin{equation*}
\xymatrix@R=15pt{
~\HFfn(K,T)~ \ar@{^(->}[r]\ar[d] & ~\HS{\FF^\n}(K,T)~ \ar^{\oplus \loc_\l^\f}[d] \\
~0~ \ar@{^(->}[r] & ~\dirsum{\l\mid\n}\Ht(K_\l,T)\otimes \Gl~
}
\end{equation*}
is cartesian.  Proposition \ref{bhprop}(i,iv) attaches to this diagram a map
$$
\wedge^{r+\nu(\n)}\HS{\FF^\n}(K,T) \otimes \wedge^{\nu(\n)}\Hom(\oplus_{\l\mid\n}\Ht(\Kl,T) \otimes \Gl,R)
   \too \wedge^r\HFfn(K,T).
$$
Tensoring both sides with $\Gn$ defines a map
$$
\Pi_{\n} : \Y{\n} \too \wedge^r\HFfn(K,T) \otimes \Gn.   
$$
\end{defn}

See the proof of Proposition \ref{ss2ks} below for an explicit description of the map $\Pi_\n$.  
Recall that if $\mm\mid\n \in \NN$, then 
$\Psi_{\n,\mm} : \Y{\n} \to \Y{\mm}$ is the map of Definition \ref{4.3}.

\begin{lem}
\label{imlem}
Suppose that hypotheses {\rm\ref{h.1}} through {\rm\ref{h.7}} 
of \S\ref{sscr} are satisfied, so in particular $r$ is the core rank of $T$.
If $\HS{(\FF^*)_{\n}}(K,T^*) = 0$ and $\mm\mid\n$, then 
$$
(\Pi_\mm \circ \Psi_{\n,\mm})(\Y{\n}) = \m^{\length(\HS{\FF(\mm)^*}(K,T^*))}\HH(\mm) = \HH'(\mm).
$$
\end{lem}

\begin{proof}
If $\HS{(\FF^*)_{\n}}(K,T^*) = 0$ then $\HS{\FF^\n}(K,T)$ is free of rank $r+\nu(\n)$ over $R$ 
by Corollary \ref{threefive}(ii).  By \eqref{gi} and \eqref{gi2} we have
$$
(\cap_{\l\mid\mm} \ker(\loc_\l^\f|\HS{\FF^\n}(K,T))) 
   \cap (\cap_{\l\mid(\n/\mm)} \ker(\loc_\l^\tr|\HS{\FF^\n}(K,T))) 
   = \HS{\FF(\mm)}(K,T).
$$
Now the lemma follows from Proposition \ref{bhprop}(ii,iii) applied to the cartesian 
square
$$
\xymatrix@R=15pt{
~\HS{\FF(\mm)}(K,T)~ \ar@{^(->}[r]\ar[d] & ~\HS{\FF^\n}(K,T)~ 
   \ar^{\oplus_{\l\mid\mm}\loc_\l^\f \oplus_{\l\mid(\n/\mm)} \loc_\l^\tr}[d] \\
~0~ \ar@{^(->}[r] & ~\dirsum{\l\mid\mm}(\Ht(K_\l,T)\otimes \Gl) \dirsum{\l\mid(\n/\mm)}\Ht(K_\l,T)~
}
$$
\end{proof}

\begin{prop}
\label{ss2ks}
Suppose $\ssys = \{\epsilon_\n : \n\in\NN\}$ is a Stark system of rank $r$ for $T$.  
Let $\Pi(\ssys)$ denote the collection $\{(-1)^{\nu(n)}\Pi_\n(\epsilon_\n) : \n\in\NN\}$.  Then:
\begin{enumerate}
\item
$\Pi(\ssys) \in \KS_r(T)$.
\item
If hypotheses {\rm\ref{h.1}} through {\rm\ref{h.7}} of \S\ref{sscr} hold, 
then $\Pi(\ssys) \in \KS_r'(T)$.
\end{enumerate}
\end{prop}

\begin{proof}
By definition $\Pi_\n(\epsilon_\n) \in \wedge^r\HFfn(K,T) \otimes \Gn$, 
so we only need to check the compatibility \eqref{compdiag}.

Suppose $\n\l \in \NN$, with $\n = \l_1 \cdots \l_{\nu(\n)}$, and for every $i$ let 
$h_i$ be a generator of $\Hom(\Ht(K_{\l_i},T),R)$ and similarly for $h$ and $\l$.  
Let 
\begin{align*}
\varpi_{\l,h}^\tr &:= \widehat{h \circ \loc_\l^\tr} 
   : \wedge^t\HS{\FF^{\n\l}}(K,T) \to \wedge^{t-1}\HS{\FF^{\n}}(K,T), \\
\varpi_{\l,h}^\f &:= \widehat{h \circ \loc_\l^\f}
  : \wedge^t\HS{\FF^{\n\l}}(K,T) \to \wedge^{t-1}\HS{\FF^{\n}(\l)}(K,T) \otimes \Gl
\end{align*}
be the maps given by Proposition \ref{wedgemap}, for $t > 0$, and similarly for 
$\varpi_{\l_i,h_i}^\tr$ and $\varpi_{\l_i,h_i}^\f$.

Let $\epsilon_{\n\l} = d_{\n\l} \otimes (h_1 \wedge \cdots \wedge h_{\nu(\n)} \wedge h)$ with 
$d_{\n\l} \in \wedge^{r+\nu(\n\l)}\HS{\FF^{\n\l}}(K,T)$, and similarly 
$\epsilon_{\n} = d_{\n} \otimes (h_1 \wedge \cdots \wedge h_{\nu(\n)})$.
By definition of $\Psi_{\n\l,\n}$ we have $d_\n = \varpi_{\l,h}^\tr(d_{\n\l})$. 
If $e$ denotes the edge joining $\n$ and $\n\l$, then
\begin{align*}
(h \otimes 1)(\psi_{\n\l}^e(\Pi_{\n\l}(\epsilon_{\n\l})) )
   &= \varpi_{\l,h}^\tr((\varpi_{\l_{1},h_{1}}^{\f} \circ \cdots \circ \varpi_{\l_{\nu(\n)},h_{\nu(\n)}}^{\f} \circ \varpi_{\l,h}^\f)(d_{\n\l})) \\
   &= (-1)^{\nu(\n)+1}(\varpi_{\l_1,h_1}^{\f} \circ \cdots \circ \varpi_{\l_{\nu(\n)},h_{\nu(\n)}}^{\f} \circ \varpi_{\l,h}^\f \circ \varpi_{\l,h}^\tr)(d_{\n\l}) \\
   &= (-1)^{\nu(\n)+1}(\varpi_{\l_1,h_1}^{\f} \circ \cdots \circ \varpi_{\l_{\nu(\n)},h_{\nu(\n)}}^{\f} \circ \varpi_{\l,h}^\f)(d_\n) \\
   &= -\varpi_{\l,h}^\f ((\varpi_{\l_1,h_1}^{\f} \circ \cdots \circ \varpi_{\l_{\nu(\n)},h_{\nu(\n)}}^{\f})(d_\n)) \\
   &= -(h \otimes 1)(\psi_\n^e(\Pi_\n(\epsilon_\n))).
\end{align*}
Since $h$ is an isomorphism, it follows that 
$\psi_{\n\l}^e(\Pi_{\n\l}(\epsilon_{\n\l})) = -\psi_\n^e(\Pi_\n(\epsilon_\n))$, 
so the collection $\{(-1)^{\nu(\n)}\Pi_\n(\epsilon_\n)\}$ is a Kolyvagin system.  
This proves (i), and (ii) follows from Lemma \ref{imlem} (using Lemma \ref{ylem}(ii)).
\end{proof}

\begin{thm}
\label{mthm3'}
If hypotheses {\rm\ref{h.1}} through {\rm\ref{h.7}} of \S\ref{sscr} hold, then 
the $R$-module map $\Pi : \SS_r(T) \to \KS'_r(T)$ of Proposition \ref{ss2ks} 
is an isomorphism.
\end{thm}

\begin{proof}
By Lemma \ref{imlem} and Theorem \ref{magic}, for every $\n$ the composition 
$$
\SS_r(T) \map{\;\Pi\;} \KS'_r(T) \too \HH'(\n)
$$
is surjective.  Since $\SS_r'(T)$ and $\KS_r'(T)$ are both free of rank one 
over $R$ (Theorems \ref{magic} and \ref{mthm3}(i)), it follows that 
$\Pi$ is an isomorphism.
\end{proof}

\section{Stub Kolyvagin systems and the dual Selmer group}
\label{sksds}

Suppose for this section that $R$ is either a principal artinian local ring or a discrete 
valuation ring.  We let $k := \length(R)$, so $k$ is finite in the artinian case 
and $k = \infty$ in the discrete valuation ring case.

Fix Selmer data $(T,\FF,\PP,r)$ satisfying hypotheses \ref{h.1} 
through \ref{h.6}, and if $R$ is artinian satisfying \ref{h.7} as well.
In this section we prove analogues for stub Kolyvagin systems of the results of 
\S\ref{extra} for Stark systems.
We will say that a stub Kolyvagin system $\ksys$ is primitive if it is primitive as 
a global section of the stub Selmer sheaf $\HH'$ (Definition \ref{perf}), i.e., if 
$\ksys$ generates the $R$-module $\KS'_r(T)$, or equivalently, if 
$\kappa_\n$ generates 
$\m^{\mesh(\n)}(\wedge^r\HFfn(K,T))\otimes\Gn$ for every $\n \in \NN$.

\begin{cor}
\label{lowerbound}
Suppose $R$ is a principal artinian ring of length $k$, and $\ksys \in \KS'_r(T)$.
\begin{enumerate}
\item
If $\kappa_1 \ne 0$ then
$$
\length(\HFs(K,T^*)) \le k - \length(R\kappa_1) 
    = \max\{i : \kappa_1 \in \m^i \wedge^r\HF(K,T)\}.$$
\item
If $\ksys$ is primitive and $\kappa_1 \ne 0$, then equality holds in {\rm(i)}.
\item
If $\ksys$ is primitive and $\kappa_1 = 0$, then $\length(\HFs(K,T^*)) \ge k$.
\end{enumerate}
\end{cor}

\begin{proof}
By Corollary \ref{threefive}(iii), $\HH'(1) = \m^{\mesh(1)}\wedge^r\HF(K,T)$ 
is a cyclic $R$-module of length 
$\max\{0,k-\length(\HFs(K,T^*))\}$.
Since $\kappa_1 \in \HH'(1)$ by definition, (i) follows.  
If $\ksys$ is primitive, then $\kappa_1$ generates $\HH'(1)$, 
which proves (ii) and (iii).
\end{proof} 

The following definition is the analogue for Kolyvagin systems of 
Definitions \ref{allde} and \ref{allder} for Stark systems.

\begin{defn}
\label{allderk}
Suppose $\ksys \in \KS_r(T)$ is a Kolyvagin system.
Define $\divity_{\ksys} \in \Maps(\NN,\Z_\ge 0 \cup \{\infty\})$ by 
$\divity_{\ksys}(\n) := \max\{j : \kappa_\n \in \mm^j\HFfns(K,T)\}$.
The {\em order of vanishing} of $\ksys$ is 
$$
\ordv{\ksys} := \min\{\nu(\n) : \n \in \NN, \kappa_\n \ne 0\} 
   = \min\{i : \partial\divity_{\ksys}(i) \ne \infty\}.
$$
We also define the  sequence of {\em elementary divisors} 
$$
d_{\ksys}(i) := \partial\divity_{\ksys}(i) - \partial\divity_{\ksys}(i+1), 
   \quad i \ge \ordv{\ksys}.
$$
\end{defn}

\begin{prop}
\label{prekork}
Suppose that $\ksys \in \KS'_r(T)$, $\ssys \in \SS_r(T)$, and $\ksys = \Pi(\ssys)$.  
Then $\ord(\ksys) = \ord(\ssys)$, 
$\partial\divity_{\ksys}(i) = \partial\divity_{\ssys}(i)$ for every $i$, 
and $d_{\ksys}(i) = d_{\ssys}(i)$ for every $i$.
\end{prop}

\begin{proof}
Suppose first that $R$ is artinian of length $k$.  Since $\Pi$ is an isomorphism 
(Theorem \ref{mthm3'}), 
we may assume  without loss of generality 
that $\ksys$ and $\ssys$ generate $\KS'_r(T)$ and $\SS_r(T)$, 
respectively.  Recall that $\smesh(\n) := \length(\HS{(\FF^*)_\n}(K,T^*))$.

For every $\n\in\NN$, Theorem \ref{mthm3}(ii) shows that 
$\kappa_\n$ generates $\m^{\mesh(\n)}\cS(\n)$, 
and Theorem \ref{magic} shows that $\epsilon_\n$ generates $\m^{\smesh(\n)}\Y{\n}$.
Thus
$$
\partial\divity_{\ksys}(i) = \begin{cases} \partial\mesh(i) & \text{if $\partial\mesh(i) < k$}, \\
   \infty & \text{if $\partial\mesh(i) \ge k$}, \end{cases} \qquad
\partial\divity_{\ssys}(i) = \begin{cases} \partial\smesh(i) & \text{if $\partial\smesh(i) < k$}, \\
   \infty & \text{if $\partial\smesh(i) \ge k$}. \end{cases}
$$
By Proposition \ref{6.12}, $\partial\mesh(i) = \partial\smesh(i)$ for every $i$, 
and all the equalities of the Proposition follow.

The case where $R$ is a discrete valuation ring follows from 
the artinian case as in the proof of Proposition \ref{11.4}.
\end{proof}

\begin{thm}
\label{korankk}
Suppose $R$ is a discrete valuation ring, $\ksys \in \KS'_r(T)$ and $\ksys \ne 0$.  Then:
\begin{enumerate}
\item
the sequence $\partial\divity_{\ksys}(t)$ is nonincreasing, and finite for $t \ge \ordv{\ksys}$, 
\item
the sequence $d_{\ksys}(i)$ is nonincreasing, nonnegative, 
and finite for $i \ge \ordv{\ksys}$, 
\item
$\ordv{\ksys}$ and the $d_{\ksys}(i)$ are independent of the choice of 
nonzero $\ksys \in \KS'_r(T)$, 
\item
$\corank_R(\HFs(K,T^*)) = \ordv{\ksys}$, 
\item
$\HFs(K,T^*)/(\HFs(K,T^*))_\dv \cong 
    \oplus_{i\ge\ordv{\ksys}} R/\m^{d_{\ksys}(i)}$,
\item
$\length_R(\HFs(K,T^*)/(\HFs(K,T^*))_\dv) 
    = \partial\divity_{\ksys}(\ordv{\ksys})-\partial\divity_{\ksys}(\infty)$, 
where $\partial\divity_{\ksys}(\infty) := \lim_{t \to \infty}\partial\divity_{\ksys}(t)$
\item
$\ksys$ is primitive if and only if $\partial\divity_{\ksys}(\infty) = 0$,
\item
$\length(\HFs(K,T^*))$ is finite if and only if $\kappa_1 \ne 0$,
\item
$
\length(\HFs(K,T^*)) \le \partial\divity_{\ksys}(0) = \max\{s : \kappa_1 \in \m^s\wedge^r\HF(K,T)\},
$
with equality if and only if $\ksys$ is primitive.
\end{enumerate}
\end{thm}

\begin{proof}
By Theorem \ref{mthm3'}, there is a (unique) $\ssys \in \SS_r(T)$ such that 
$\Pi(\ssys) = \ksys$.  By Proposition \ref{prekork}, all the invariants of 
Definition \ref{allderk} attached to $\ksys$ are equal to the corresponding 
invariants of $\ssys$.  Now the theorem follows from Theorem \ref{korank}.
\end{proof}

\section{Proof of Theorem \ref{mthm0}}
\label{m0pf}

Keep the notation of \S\ref{stubks}, so $R$ is principal and artinian of length $k$, 
hypotheses \ref{h.1} through \ref{h.7} hold.  In particular we assume 
that $r = \cm(T)$, the core rank of $T$.  

\begin{lem}
\label{9.1}
The sheaf $\HH'$ is locally cyclic.
\end{lem}

\begin{proof}
By Corollary \ref{threefive}(iii), for every $\n\in\NN$ the stalk
$\HH'(\n)$ is a cyclic $R$-module.
By Definition \ref{stubsheaf} and Proposition \ref{eqim} the vertex-to-edge maps $\pne$ are all 
surjective, and so the edge stalks $\HH'(e)$ are all cyclic as well.
\end{proof}

\begin{lem}
\label{9.2}
Suppose $\n$ is a core vertex, and $\l \in \PP$ does not divide $\n$.  
Let $e$ denote the edge joining $\n$ and $\n\l$.  
Then the following are equivalent:
\begin{enumerate}
\item
$\loc_\l : \HFfn(K,T)[\m] \to \Hf(K_\l,T)$ is nonzero,
\item
$\n\l$ is a core vertex and 
both maps $\pne : \HH(\n) \to \HH(e)$, $\pnle : \HH(\n\l) \to \HH(e)$ 
are isomorphisms.
\end{enumerate}
\end{lem}

\begin{proof}
Suppose that (i) holds.
Since $\Il = 0$ by \ref{h.7}, Lemma \ref{isofunct}(ii) shows that $\Hf(K_\l,T)$ is free of rank one over $R$.
Since $\n$ is a core vertex, $\HFfn(K,T)$ is a free $R$-module of rank $r$.  In particular 
$\HFfn(K,T)[\m] = \m^{k-1}\HFfn(K,T)$, and it follows that the localization map 
$\loc_\l : \HFfn(K,T) \to \Hf(K_\l,T)$ is surjective.  By Proposition \ref{wedgemap}, 
it follows that $\pne$ is an isomorphism. 

Further, since $\loc_\l : \HFfn(K,T) \to \Hf(K_\l,T)$ is surjective, and $\Ht(K_\l,T)$ is 
free of rank one over $R$, and $\HFfns(K,T^*) = 0$, 
\cite[Lemma 4.1.6]{kolysys} shows that $\n\l$ is a core vertex and 
$\loc_\l : \HFfnl(K,T) \to \Ht(K_\l,T)$ is surjective.  Now Proposition \ref{wedgemap}
shows that that $\pnle$ is an isomorphism.  Thus (ii) holds.

Conversely, if $\pne$ is an isomorphism then Proposition \ref{wedgemap} shows that the map
$\loc_\l : \HFfn(K,T) \to \Hf(K_\l,T)$ is surjective, and since $\Hf(K_\l,T)$ is free of rank 
one over $R$ it follows that $\loc_\l$ is not identically zero on $\HFfn(K,T)[\m]$.  Thus 
(ii) implies (i).
\end{proof}

Recall that $\Tb := T/\m T$.

\begin{prop}
\label{9.4}
Suppose $\n \in \NN$ and $\mesh(\n,\Tb^*) > 0$.  Then there is a $\l \in \PP$ prime to $\n$ 
such that $\mesh(\n\l,\Tb^*) < \mesh(\n,\Tb^*)$ and 
$\pne : \HH'(\n) \to \HH'(e)$ is an isomorphism, where $e$ is the edge joining $\n$ 
and $\n\l$. 
\end{prop}

Let $\bar\mesh(\n) := \dim_\k\HFfns(K,\Tb^*)$. By Proposition \ref{stprops}(ii), 
we have $\mesh(\n) = 0$ if and only if $\bar\mesh(\n) = 0$.

\begin{proof}
By \cite[Proposition 3.6.1]{kolysys} we can use the Cebotarev theorem to 
choose a prime $\l \in \PP$ such that the localization maps 
$$
\m^{k-1}\HS{\FFf{\n}}(K,T) \to \Hf(K_\l,T), \quad \HS{\FFf{\n}^*}(K,T^*)[\m] \to \Hf(K_\l,T^*)
$$
are both nonzero.  
(Note that $\m^{k-1}\HFfn(K,T) \ne 0$ by Corollary \ref{threefive}(iii).)  
Then by Poitou-Tate global duality (see for example \cite[Lemma 4.1.7(iv)]{kolysys}), 
we have $\bar\mesh(\n\l) < \bar\mesh(\n)$.
Further, we have that 
localization $\HS{\FFf{\n}}(K,T) \to \Hf(K_\l,T)$ is surjective, 
so by Proposition \ref{wedgemap}(ii)
$$
\widehat{\loc_\l} : \wedge^r \HS{\FFf{\n}}(K,T) 
   \too \Hf(K_\l,T) \otimes (\wedge^{r-1} \HS{\FF_\l(\n)}(K,T))
$$ 
is surjective as well.  Since $\HH'(e)$ is defined to be the image of 
$$
\HH'(\n) := \m^{\mesh(\n)}(\wedge^r \HS{\FFf{\n}}(K,T)) \otimes \Gn
$$
under the upper maps of \eqref{compdiag}, we deduce that 
$$
\HH'(e) = \m^{\mesh(\n)} \Ht(K_\l,T) \otimes (\wedge^{r-1} \HS{\FF_\l(\n)}(K,T)) \otimes \Gnl.
$$
Thus 
$$
\length_R(\HH'(e)) \ge k - \mesh(\n) = \length_R(\HH'(\n)),
$$ 
the equality by Corollary \ref{threefive}(iii).  Since the map $\HH'(\n) \to \HH'(e)$ 
is surjective by definition, it must be an isomorphism.
\end{proof}

\begin{thm}
\label{14.4}
Suppose $\n,\n'$ are core vertices.  Then there is a path
$$
\xymatrix{
\n=\n_0 \ar@{-}^-{e_1}[r] & \n_1 \ar@{-}^-{e_2}[r] & \cdots \ar@{-}^-{e_t}[r] & \n_t=\n'
}
$$
in $\X$ such that every $\n_i$ is a core vertex and 
all of the maps $\psi_{\n_i}^{e_{i+1}}$ and $\psi_{\n_i}^{e_{i}}$ 
are isomorphisms.
\end{thm}

\begin{proof}
When $\cm(T)=1$, this is \cite[Theorem 4.3.12]{kolysys}.  The general case can be proved 
in the same way, but instead we will prove it here by induction on $r := \cm(T)$.

Denote by $\bar\FF$ the induced Selmer structure on $\Tb$.  
By Proposition \ref{stprops} and the definition of core vertices we see that 
the Selmer sheaves $\HH_{(T,\FF,\PP)}$ and $\HH_{(\Tb,\bar\FF,\PP)}$ 
have the same core vertices and the same core rank $r$ (see also \cite[Theorem 4.1.3]{kolysys}).
  
Since $r > 0$, we can fix nonzero classes $c \in \HFfn(K,\Tb)$ and $c' \in \HS{\FF(\n')}(K,\Tb)$.  
By \cite[Proposition 3.6.1]{kolysys}, we can use the Cebotarev theorem to 
choose $\l \in \PP$, not dividing $\n\n'$, such that 
the localizations $c_\l$ and $c'_\l$ are both nonzero.

Note that the Selmer triple $(\Tb,\bar\FF_\l,\PP-\{\l\})$ also 
satisfies hypotheses \ref{h.1} through \ref{h.6} 
(the only one of those conditions that depends on the Selmer structure is \ref{h.5}, and \ref{h.5} 
is vacuous when we work over $R/\m$).  By our choice of $\l$, both localization maps 
$$
\loc_\l : \HS{\bar\FF(\n)}(K,\Tb) \to \Hf(K_\l,\Tb), \qquad \loc_\l : \HS{\bar\FF(\n')}(K,\Tb) \to \Hf(K_\l,\Tb)
$$
are nonzero, and $\Hf(K_\l,\Tb)$ is a one-dimensional $R/\m$-vector space, so both maps are 
surjective.  Since $\n$ and $\n'$ are core vertices for $(\Tb,\bar\FF)$, it follows that 
$$
\dim_{R/\m}\HS{\bar\FF_\l(\n)}(K,\Tb) = \dim_{R/\m}\HS{\bar\FF_\l(\n')}(K,\Tb) = r-1
$$
and (by Poitou-Tate global duality, see for example \cite[Theorem 2.3.4]{kolysys}) 
that 
$\HS{\bar\FF_\l(\n)^*}(K,\Tb^*) = \HS{\bar\FF_\l(\n')^*}(K,\Tb^*) = 0.$

In particular we deduce that $\cm(\Tb,\bar\FF_\l) = r-1$, and that $\n, \n'$ are core vertices 
for the sheaf $\HH_{\Tb,\bar\FF_\l}$.  By our induction hypotheses we conclude that there is a 
path $\n=\n_0,\n_1,\ldots,\n_t=\n'$ from $\n$ to $\n'$ in $\X$ such that every $\n_i$ is 
prime to $\l$, every $\n_i$ is a core vertex for $\HH_{\Tb,\bar\FF_\l}$, 
and every vertex-to-edge map (for $\HH_{\Tb,\bar\FF_\l}$) along the path is an isomorphism. 
We will show that every $\n_i$ is a core 
vertex for $\HH_{T,\FF}$, and every vertex-to-edge map (for $\HH_{T,\FF}$) 
along the path is an isomorphism.  This will prove the theorem.

Fix $i$, $0 \le i \le t$.  The exact sequence
$$
0 \too \HS{\bar\FF_\l(\n_i)}(K,\Tb) \too \HS{\bar\FF(\n_i)}(K,\Tb) \map{\loc_\l} 
   \Hf(K_\l,\Tb) 
$$
shows that $\dim_{R/\m}\HS{\bar\FF(\n_i)}(K,\Tb) \le r$.  Then Corollary \ref{threefive}(i) 
(applied to $\Tb$, $\bar\FF$, and $R/\m$) shows that $\n_i$ is a core vertex of $\HH_{\Tb,\bar\FF}$, 
and hence is a core vertex of $\HH_{T,\FF}$.

Further, suppose $\el$ is a prime such that $\n_{i \pm 1} = \n_i\el$, and let $e$ be the edge joining 
those two vertices.  By assumption, the maps $\HH_{\Tb,\bar\FF_\l}(\n_i) \to \HH_{\Tb,\bar\FF_\l}(e)$ 
and $\HH_{\Tb,\bar\FF_\l}(\n_i\el) \to \HH_{\Tb,\bar\FF_\l}(e)$
are isomorphisms, so by Lemma \ref{9.2} the localization map 
$\HS{\bar\FF_\l(\n_i)}(K,\Tb) \to \Hf(K_\el,\Tb)$ is nonzero.  But 
$$
\HS{\bar\FF_\l(\n_i)}(K,\Tb) \subset \HS{\bar\FF(\n_i)}(K,\Tb) = \HS{\FF(\n_i)}(K,T)[\m],
$$
so $\loc_\el : \HS{\FF(\n_i)}(K,T)[\m] \to \Hf(K_\el,T)$ is nonzero, so by Lemma \ref{9.2} 
both of the maps $\psi_{n_i}^e$ and $\psi_{n_{i\pm 1}}^e$ are isomorphisms.
This completes the proof.
\end{proof}

\begin{cor}
\label{9.5}
There are core vertices.  More precisely:
\begin{enumerate}
\item
for every $\n\in\NN$ there is an $\n'\in\NN$ prime to $\n$, with $\nu(\n') = \bar\mesh(\n)$, 
such that $\n\n'$ is a core vertex,
\item
$\min\{\nu(\n) : \text{$\n$ is a core vertex}\} = \dim_{R/\m}\HFs(K,T^*)[\m].$
\end{enumerate}
\end{cor}

\begin{proof}
Choose $\n \in \NN$.  For every $\n' \in\NN$ prime to $\n$, global duality 
(see for example \cite[Lemma 4.1.7(i)]{kolysys}) shows that
\begin{equation}
\label{nts}
\bar\mesh(\n\n') \ge \bar\mesh(\n)-\nu(\n').
\end{equation}

Applying Proposition \ref{9.4}, we can construct $\n=\n_0, \n_1, \n_2, \ldots \in \NN$ 
inductively, with $\n_{i+1} = \n_i\l_i$ for some prime $\l_i \in \NN$ and 
$\bar\mesh(\n_{i+1}) < \bar\mesh(\n_i)$, 
until we reach $\n_d \in\NN$ with $\bar\mesh(\n_d) = 0$.  Then 
$\HS{\FF(\n_d)^*}(K,T^*)[\m] = \HS{\FF(\n_d)^*}(K,\Tb^*) = 0$, so $\n_d$ is a core vertex.
Setting $\n' := \n_d/\n$ we have 
$$
\nu(\n') = d \le \bar\mesh(\n) = \dim_{R/\m}\HFs(K,T^*)[\m].
$$
By \eqref{nts}, since $\bar\mesh(\n\n') = 0$ we have $\nu(\n') \ge \bar\mesh(\n)$, 
and so $\nu(\n') = \bar\mesh(\n)$.  This proves (i), and applying (i) 
with $\n = 1$ and \eqref{nts} proves (ii).
\end{proof}

\begin{proof}[Proof of Theorem \ref{mthm0}]
Theorem \ref{mthm0}(i) is Corollary \ref{9.5}, and Theorem \ref{mthm0}(ii) is Theorem \ref{14.4}.
Lemma \ref{9.1} says that $\HH'$ is locally cyclic.  
To complete the proof of Theorem \ref{mthm0} we need only show 
that every core vertex is a hub of $\HH'$.

Fix a core vertex $\n_0$, and let $\n\in\NN$ be any other vertex.  We will show 
by induction on $\bar\mesh(\n)$ that there is an $\HH'$-surjective path from $\n_0$ to $\n$.  

If $\bar\mesh(\n) = 0$, then $\n$ is also a core vertex and the desired surjective path 
exists by Theorem \ref{14.4}.

Now suppose $\bar\mesh(\n) > 0$.  Use Proposition \ref{9.4} to find $\l \in \PP$ 
not dividing $\n$ such that $\bar\mesh(\n\l) < \bar\mesh(\n)$ and 
$\pne : \HH'(\n) \to \HH'(e)$ is an isomorphism, where $e$ is the edge joining $\n$ 
and $\n\l$.  By induction there is an $\HH'$-surjective path from $\n_0$ to $\n\l$, and if we 
adjoin to that path the edge $e$, we get an $\HH'$-surjective path from $\n_0$ to $\n$.
\end{proof}

\appendix

\section{Some exterior algebra}
\label{exalg}

Suppose for this appendix that $R$ is a local principal ideal ring 
with maximal ideal $\m$.

\begin{prop}
\label{wedgemap}
Suppose $0 \to N \to M \map{\psi} C$ is an exact sequence of finitely-generated $R$-modules, 
with $C$ free of rank one, and $r \ge 1$.  Then there is a unique map 
$$
\hat\psi : \wedge^r M \too C \otimes \wedge^{r-1}N
$$
such that 
\begin{enumerate}
\item
the composition 
$\wedge^r M \map{\hat\psi} C \otimes \wedge^{r-1}N \to C \otimes \wedge^{r-1}M$
is given by 
$$
m_1\wedge\cdots\wedge m_r \mapsto 
   \sum_{i=1}^r (-1)^{i+1} \psi(m_i) \otimes (m_1\wedge\cdots\wedge m_{i-1}\wedge m_{i+1}\cdots\wedge m_r),
$$
\item
the image of $\hat\psi$ is the image of $\psi(M) \otimes \wedge^{r-1}N \to C \otimes \wedge^{r-1}N$.
\end{enumerate}
If $M$ is free of rank $r$ over $R$, then $\hat\psi$ is an isomorphism if and only if 
$\psi$ is surjective.
\end{prop}

\begin{proof}
Since $R$ is principal, we can ``diagonalize'' $\psi$ and write $M = Rm \oplus N_0$ and 
$N = Im \oplus N_0$ where $N_0 \subset N$, $m \in M$ is such that $\psi(m)$ generates  
$\psi(M)$, and $I$ is an ideal of $R$.  In particular we have $0 = \psi(N) = I\psi(M)$.

The formula of (i) gives a well-defined $R$-module homomorphism 
$\hat\psi_0 : \wedge^r M \to \psi(M) \otimes \wedge^{r-1}M$.  
Consider the diagram
$$
\xymatrix{
\wedge^r M \ar^-{\hat\psi_0}[r] & \psi(M) \otimes \wedge^{r-1}M \ar[r] &C \otimes \wedge^{r-1}M \\
& \psi(M) \otimes \wedge^{r-1}N \ar^-{\eta_2}[r]\ar_{\eta_1}[u] & C \otimes \wedge^{r-1}N \ar[u]
}
$$
with maps induced by the inclusions $\psi \hookto C$ and $N \hookto M$.  We will show 
that $\image(\hat\psi_0) \subset \image(\eta_1)$ and $\ker(\eta_1) \subset \ker(\eta_2)$.  
Then $\hat\psi := \eta_2 \circ \eta_1^{-1} \circ \hat\psi_0$ is well defined and 
satisfies (i) and (ii).

Since $M = Rm \oplus N_0$, we have that the image $\hat\psi_0(\wedge^r M)$ is generated by monomials 
$\psi(m) \otimes n_1 \wedge \cdots \wedge n_{r-1}$ with $n_i \in N_0$, so $\image(\hat\psi_0) \subset \image(\eta_1)$.

We also have 
\begin{align*}
\wedge^{r-1}N &= (Im \otimes \wedge^{r-2}N_0) \oplus \wedge^{r-1}N_0, \\
\wedge^{r-1}M &= (Rm \otimes \wedge^{r-2}N_0) \oplus \wedge^{r-1}N_0.
\end{align*}
Therefore, since $I\psi(M) = 0$,
\begin{multline*}
\ker(\eta_1) = \ker(\psi(M) \otimes Im \otimes \wedge^{r-2}N_0 \to \psi(M) \otimes Rm \otimes \wedge^{r-2}N_0) \\
   \hskip2.75in= \psi(M) \otimes Im \otimes \wedge^{r-2}N_0.
\end{multline*}
We further have
\begin{equation}
\label{(a.1)}
\eta_2(\psi(M) \otimes Im \otimes \wedge^{r-2}N_0) = 0.
\end{equation}
Thus $\ker(\eta_1) \subset \ker(\eta_2)$, so $\hat{\psi}$ is well-defined and 
has properties (i) and (ii).  
Uniqueness follows from the fact that by \eqref{(a.1)}
$$
\eta_2(\psi(M) \otimes \wedge^{r-1}N) = \eta_2(\psi(M) \otimes \wedge^{r-1}N_0)
$$
injects into $C \otimes \wedge^{r-1}M$.

The final assertion follows easily from the definition of $\hat\psi$ above.
\end{proof}

If $M$ is an $R$-module, let $\dual{M} := \Hom(M,R)$.

\begin{prop}
\label{bhprop}
Suppose $R$ is artinian and there is a cartesian diagram of $R$-modules
$$
\xymatrix@R=15pt{
M_1 \ar@{^(->}[r]\ar[d] & M_2 \ar^{h}[d] \\
C_1 \ar@{^(->}[r] & C_2
}
$$
where $C_1$ and $C_2$ are free $R$-modules of finite rank, 
and the horizontal maps are injective.
\begin{enumerate}
\item
Suppose $r \ge 0$ and  $s_i = \rank_R(C_i)$.
There is a canonical $R$-module homomorphism 
$$
\wedge^{r+s_2}M_2 \otimes \wedge^{s_2} \dual{C}_2 \too 
   \wedge^{r+s_1}M_1 \otimes \wedge^{s_1} \dual{C}_1
$$
defined as follows.  If $m \in \wedge^{r+s_2}M_2$, $\psi_1,\ldots, \psi_{s_2}$ 
is a basis of $\dual{C}_2$ such that $\psi_{s_1+1},\ldots, \psi_{s_2}$ is a basis of 
$\dual{(C_2/C_1)}$, and $h_i = \psi_i \circ h$, then
$$
m \otimes (\psi_{1}\wedge\cdots\wedge\psi_{s_2}) \mapsto
   (\hat{h}_{s_1+1} \circ \cdots \circ \hat{h}_{s_2})(m)
      \otimes (\psi_{1}\wedge\cdots\wedge\psi_{s_1})
$$
with $\hat{h}_i$ as in Proposition \ref{wedgemap}.
This is independent of the choice of the $\psi_i$.
\item
If $M_2$ is free of rank $r+s_2$ over $R$, then the image of the map of (i) 
is $$\m^{\length(M_1) - (r+s_1)\length(R)}\wedge^{r+s_1}M_1 \otimes \wedge^{s_1} \dual{C}_1.$$
\item
If 
$$
\xymatrix@R=15pt{
M_2 \ar@{^(->}[r]\ar[d] & M_3 \ar[d] \\
C_2 \ar@{^(->}[r] & C_3
}
$$
is another such cartesian square, then the triangle
$$
\xymatrix@C=8pt{
\wedge^{r+s_3}M_3 \otimes \wedge^{s_3} \dual{C}_3 
   \ar[rr]\ar[dr] && \wedge^{r+s_1}M_1 \otimes \wedge^{s_1} \dual{C}_1 \\
   & \wedge^{r+s_2}M_2 \otimes \wedge^{s_2} \dual{C}_2 \ar[ur]
}
$$
induced by the maps of (i) commutes.
\item
Suppose there is an exact sequence $0 \to M_1 \to M_2 \to C$, where $C$ is free  
of rank $s$ over $R$.  
Then for every $r \ge 0$, the map of (i) (with $C_1 = 0$ and $C_2 = C$) is a canonical map 
$\wedge^{r+s}M_2 \otimes \wedge^s\dual{C} \to \wedge^r M_1$.
\end{enumerate}
\end{prop}

\begin{proof}
Since the square is cartesian, and by our choice of the $\psi_i$, we have 
\begin{equation}
\label{(14)}
\ker(\oplus_{i > s_1} h_i) = h^{-1}(C_1) = M_1.
\end{equation}
Applying Proposition \ref{wedgemap} repeatedly shows that the map defined in (i) takes values in 
$\wedge^{r+s_1}M_1 \otimes \wedge^{s_1} \dual{C}_1$.
It is straightforward to check that this map is independent of the choice of 
the $\psi_i$.  This proves (i), and (iv) is just a special case of (i).

Suppose now that $M_2$ is free of rank $r+s_2$, and let $s := s_2-s_1$.
 Choose an $R$-basis $\eta_1,\ldots,\eta_{r+s_2}$ of $\dual{M_2}$ such that 
the span of $\eta_{1},\ldots,\eta_{s}$ contains $h_{s_1+1},\ldots,h_{s_2}$,
i.e., there is an $s \times s$ matrix $A = [a_{ij}]$ with $a_{ij} \in R$ such that 
$h_{s_1+j} = \sum_{i}a_{ij}\eta_i$.  
Let $N := \cap_{i=1}^{s} \ker(\eta_i)$.
Then $N$ is free over $R$ of rank $r+s_1$,  
and we have a split exact sequence of free modules
$$
0 \too N \too M_2 \map{\oplus_{i\le s} \eta_i} R^{s} \too 0.
$$
It follows that the composition 
$\hat\eta_{1}\circ\cdots\circ\hat\eta_{s} : \wedge^{r+s_2}M_2 \to \wedge^{r+s_1}N$
of maps given by Proposition \ref{wedgemap} is an isomorphism.

We also have 
$$
\hat{h}_{s_1+1}  \cdots \circ \hat{h}_{s_2}
  = \det(A) \; \hat{\eta}_1 \circ \cdots \circ \hat{\eta}_s,
$$
and $N \subset M_1$ by \eqref{(14)}.  Since $N$ is free, 
there is a noncanonical splitting 
$$
M_1 \cong N \oplus M_1/N,
$$
so the map
$$
\m^{\length(M_1/N)}\wedge^{r+s_1}N \too \m^{\length(M_1/N)}\wedge^{r+s_1}M_1
$$
induced by the inclusion $N \hookto M_1$ is surjective.  Finally,
$$
\det(A)R = \m^{\length(M_1/N)} = \m^{\length(M_1)-(r+s_1)\length(R)},
$$
and combining these facts proves (ii).

Assertion (iii) follows from the independence of the choice of the $\psi_i$.
Choose a basis  $\psi_1,\ldots, \psi_{s_s}$ of $\dual{C}_3$ such that 
$\psi_{s_1+1},\ldots, \psi_{s_3}$ is a basis of $\dual{(C_3/C_1)}$ and 
$\psi_{s_2+1},\ldots, \psi_{s_3}$ is a basis of $\dual{(C_3/C_2)}$.  Then 
$\psi_{s_1+1}|_{C_2},\ldots, \psi_{s_2}|_{C_2}$ is a basis of $\dual{(C_2/C_1)}$, 
and (iii) just reduces to the statement that
$$
(\hat{\psi}_{s_1+1} \circ \cdots \circ \hat{\psi}_{s_2}) 
   \circ (\hat{\psi}_{s_2+1} \circ \cdots \circ \hat{\psi}_{s_3})
   = (\hat{\psi}_{s_1+1} \circ \cdots \circ \hat{\psi}_{s_3}).
$$
\end{proof}

\medskip
\subsection*{\em Erratum to \cite{kolysys}}
We thank Cl\'ement Gomez for pointing out an error in the statement of 
\cite[Lemma 2.1.4]{kolysys}.  The correct statement (which is 
all that was used elsewhere in \cite{kolysys}) should be:

\begin{mlem}
If $(T/\mm T)^{G_\Q} = 0$ then $(T/I T)^{G_\Q} = 0$ for every ideal $I$ of $R$.
\end{mlem}


\begin{thebibliography}{MR2}
\bibitem[MR1]{kolysys}
   B.\ Mazur, K.\ Rubin, Kolyvagin systems. {\em Memoirs of the Amer.\ Math.\ Soc.} 
   {\bf 799} (2004).
   {\em Compositio Math.} {\bf 147} (2011) 56--74.
\bibitem[MR2]{gen.darmon}
   B.\ Mazur, K.\ Rubin, Refined class number formulas for $\mathbb{G}_m$.  To appear.
   \url{http://math.uci.edu/~krubin/preprints/generaldc.pdf}
\bibitem[Mi]{milne} 
   J.S.\ Milne, Arithmetic duality theorems. {\em 
   Perspectives in Math.}\ {\bf 1}, Orlando: Academic Press (1986).
\bibitem[PR1]{bprheights} 
   B.\ Perrin-Riou, Th\'eorie d'Iwasawa et hauteurs $p$-adiques.
     {\em Invent. Math.}\ {\bf 109} (1992) 137--185.
\bibitem[PR2]{pr.es} 
   B.\ Perrin-Riou, Syst\`emes d'Euler $p$-adiques et th\'eorie d'Iwasawa. 
   {\em Ann. Inst. Fourier (Grenoble)} {\bf 48} (1998) 1231--1307.
\bibitem[Ru1]{rubin-stark} 
   K.\ Rubin, A Stark conjecture ``over $\Z$'' for abelian $L$-functions 
   with multiple zeros.  {\em Ann.\ Inst.\ Fourier (Grenoble)} {\bf 46} 
   (1996) 33--62.
\bibitem[Ru2]{EulerSystems}
   K.\ Rubin, Euler Systems.  {\em Annals of Math. Studies} {\bf 147},
   Princeton: Princeton University Press (2000).
\bibitem[S]{sano}
   T.\ Sano, A generalization of Darmon's conjecture for Euler systems for 
   general $p$-adic representations.  To appear.
\bibitem[T]{tate}
  J.\ Tate, Les Conjectures de Stark sur les Fonctions L d'Artin en $s=0$. {\em Prog.\ in Math.} {\bf 47}, 
  Birkh\"auser, Boston-Basel-Stuttgart (1984). 
\bibitem[Wi]{wiles} 
   A.\ Wiles, Modular elliptic curves and Fermat's Last Theorem.
   {\em Annals of Math.} {\bf 141} (1995) 443--551.
\end{thebibliography}
\end{document}